\RequirePackage{fix-cm}
\documentclass[smallextended,envcountsect]{svjour3}       
\smartqed  
\usepackage{graphicx}

\usepackage{bm}
\usepackage{tikz,mathpazo}
\usepackage{pgf}
\usepackage[top=2.5cm, bottom=2.5cm, left=3cm, right=3cm]{geometry}   
\usepackage{indentfirst}
\usepackage{graphicx}
\usepackage{graphics}
\usepackage[toc,page,title,titletoc,header]{appendix}
\usepackage{bm}
\usepackage{bbm}
\usepackage{listings}  
\usepackage{amsmath}
\usepackage{setspace} 
\usepackage{indentfirst}
\usepackage{caption}
\usepackage{multirow} 
\usepackage{lipsum,multicol}
\usepackage{pdfpages}
\usepackage{float}
\usepackage{pdfpages}
\usepackage{url}
\usepackage{colortbl}
\usepackage{subfigure}
\usepackage{epsfig}
\usepackage{epstopdf}
\usetikzlibrary{shapes.geometric, arrows}
\usepackage{fancyhdr}
\usepackage{abstract}
\usepackage{tikz,mathpazo}
\usetikzlibrary{shapes.geometric, arrows}
\usepackage{amssymb}
\usepackage{latexsym}
\usepackage{verbatim}
\usepackage[numbers]{natbib} 
\usepackage{booktabs}

\allowdisplaybreaks[4]

\newcommand{\inner}[1]{\left\langle #1 \right\rangle}
\newcommand{\norm}[1]{\left\Vert #1\right\Vert}
\newcommand{\bb}[1]{\mathbb{#1}}

\newcommand{\conv}[0]{\mathrm{conv}}
\newcommand{\cl}[0]{\mathrm{cl}\,} 

\newcommand{\ca}[1]{\mathcal{#1}}

\newcommand{\M}[0]{\mathcal{M}}

\newcommand{\mk}{{m_{k} }}

\newcommand{\mkp}{{m_{k+1} }}

\newcommand{\xk}{{x_{k} }}
\newcommand{\yk}{{y_{k} }}

\newcommand{\zk}{{z_{k} }}
\newcommand{\xkp}{{x_{k+1} }}

\newcommand{\zkp}{{z_{k+1} }}

\newcommand{\D}{\ca{D}}

\newcommand{\grad}{{\mathit{grad}\,}}

\newcommand{\Rn}{\mathbb{R}^n}
\newcommand{\Rp}{\mathbb{R}^p}
\newcommand{\Rm}{\mathbb{R}^m}

\usepackage{hyperref}
\hypersetup{
	colorlinks=true,
	linkcolor=blue,
	filecolor=magenta,      
	urlcolor=blue,
	citecolor=blue,
}

\newtheorem{theo}{Theorem}[section]
\newtheorem{lem}[theo]{Lemma}
\newtheorem{prop}[theo]{Proposition}

\newtheorem{defin}[theo]{Definition}
\newtheorem{rmk}[theo]{Remark}
\newtheorem{assumpt}[theo]{Assumption}

\usepackage[figuresright]{rotating}
\usepackage{pdflscape}

\usepackage{caption}
\usepackage{algorithm}
\usepackage{algpseudocode}
\usepackage{longtable}
\usepackage{appendix}

\numberwithin{equation}{section}

\usepackage{array}

\begin{document}

\title{Learning-rate-free Momentum SGD with Reshuffling Converges in Nonsmooth Nonconvex Optimization\thanks{The research of Xiaoyin Hu is supported by Zhejiang Provincial Natural Science Foundation of China under Grant (No. LQ23A010002), the National Natural Science Foundation of China (Grant No. 12301408), Scientific Research Foundation of Hangzhou City University(No.J-202317), and the advanced computing resources provided by the Supercomputing Center of HZCU. The work of Xin Liu is supported in part by the National Key R\&D Program of China (2023YFA1009300), the National Natural Science Foundation of China (12125108, 12226008, 11991021, 11991020, 12021001, 12288201), Key Research Program of Frontier Sciences, Chinese Academy of Sciences (ZDBS-LY-7022), and CAS AMSS-PolyU Joint Laboratory of Applied Mathematics.}
}


\author{Xiaoyin Hu  		\and
        Nachuan Xiao        \and
        Xin Liu             \and
        Kim-Chuan Toh		
}


\institute{Xiaoyin Hu \at
              School of Computer and Computing Science, Hangzhou City University, Hangzhou, 310015, China.\\
              \email{hxy@amss.ac.cn} 
          \and
          Nachuan Xiao \at
              Institute of Operational Research and Analytics, National University of Singapore, Singapore. \\
              \email{xnc@lsec.cc.ac.cn}           
           \and
          Xin Liu \at
              State Key Laboratory of Scientific and Engineering Computing, Academy of Mathematics and Systems Science, Chinese Academy of Sciences, and University of Chinese Academy of Sciences, China.\\
              \email{liuxin@lsec.cc.ac.cn}
          \and 
          Kim-Chuan Toh \at
              Department of Mathematics, and Institute of Operations Research and Analytics, National University of Singapore, Singapore 119076.\\
              \email{mattohkc@nus.edu.sg} 
}

\date{Received: date / Accepted: date}

\maketitle

\begin{abstract}
In this paper, we propose a generalized framework for developing learning-rate-free momentum stochastic gradient descent (SGD) methods in the minimization of nonsmooth nonconvex functions, especially in training nonsmooth neural networks. Our framework adaptively generates learning rates based on the historical data of stochastic subgradients and iterates. Under mild conditions, we prove that our proposed framework enjoys global convergence to the stationary points of the objective function in the sense of the conservative field, hence providing convergence guarantees for training nonsmooth neural networks. Based on our proposed framework, we propose a novel learning-rate-free momentum SGD method (LFM). Preliminary numerical experiments reveal that LFM performs comparably to the state-of-the-art learning-rate-free methods (which have not been shown theoretically to be convergence) across well-known neural network training benchmarks.
\end{abstract}

\section{Introduction}

In this paper, we consider the following unconstrained optimization problem,
\begin{equation}
    \tag{UOP}
    \label{Prob_Ori}
    \min_{x \in \Rn}\quad f(x) := \frac{1}{N}\sum_{i = 1}^N f_i(x),
\end{equation}
where  $f_i: \Rn \to \bb{R}$ is locally Lipschitz continuous, nonconvex, and possibly nonsmooth for each $i \in [N] := \{1, 2,\ldots, N\}$. Optimization problems that take the form of \eqref{Prob_Ori} have wide applications in various areas of deep learning, such as computer vision \cite{wei2021end,jin2022roby,wang2023chromosome}, natural language processing \cite{fu2021end}, and reinforcement learning \cite{chen2023symbolic}, urban computing\cite{zhou2022unsupervised, zhou2023tempo}.

Stochastic optimization methods play a fundamental role in deep learning, exhibiting efficient performance in solving large-scale optimization problems. The selection of suitable learning rates (i.e., stepsizes) is a critical aspect of these methods, particularly in practical scenarios with limited prior knowledge about the underlying optimization problems. Towards that challenge, several recent works have developed a variety of adaptive methods, including ADAM \cite{kingma2014adam}, AdaGrad \cite{duchi2011adaptive}, and RMSProp \cite{hinton2012neural}, which dynamically employ suitable preconditioners for the update directions based on the noise level of each coordinate. However, these adaptive methods still necessitate the tuning of their learning rates, and their numerical performance highly depends on the choices of the learning rates. To the best of our knowledge, how to efficiently choose appropriate learning rates for these methods remains an open question. Additionally, the expressions of the theoretically optimal learning rates in these existing works rely on some unknown prior knowledge of \eqref{Prob_Ori}, including the distance between the initial point and the optimal solution set, the Lipschitz constant of the objective function, and the initial optimality gap on the function value.

In light of the challenges in choosing appropriate learning rates,  {\it learning-rate-free} methods are proposed to automatically adjust learning rates for stochastic optimization methods with almost no prior knowledge of the optimization problems \cite{streeter2012no}.  Most of the existing learning-rate-free optimization approaches are specifically designed for online stochastic convex optimization, and can achieve near-optimal convergence rates under appropriate conditions \cite{luo2015achieving,orabona2016coin,cutkosky2018black,mhammedi2020lipschitz,bhaskara2020online,jacobsen2022parameter,zhang2022pde}. Moreover, within the framework of the stochastic gradient descent (SGD), numerous studies \cite{orabona2014simultaneous,orabona2017training,kempka2019adaptive,chen2022better,carmon2022making,defazio2023learning,ivgi2023dog,mishchenko2023prodigy} develop practical learning-rate-free SGD methods for convex optimization. In addition, given that learning rates are the only parameters in SGD, all these aforementioned learning-rate-free methods are usually referred to as \textit{parameter-free} methods. As there are always other parameters involved in these methods (e.g., the parameters for momentum terms and variance estimators in \cite{defazio2023learning,mishchenko2023prodigy}), the word ``learning-rate-free'' might be more appropriate than ``parameter-free''.  
These methods dynamically adjust their learning rates by utilizing the historical stochastic gradients across iterations. In particular, \cite{ivgi2023dog} propose a practical learning-rate-free method named {\it Distance over Gradients} (DoG). DoG leverages the update schemes from SGD, and automatically chooses its learning rates based on the distance from the initial point and norms of gradients, as described by the following rule,
\begin{equation}
\tag{DoG}
    \xkp = \xk - \eta_k g_k, \qquad \text{where} \quad \eta_k = \rho \cdot \frac{\max_{i\leq k}\norm{x_i - x_0}}{\sqrt{\varepsilon_0 + \sum_{i = 0}^k \norm{g_i}^2}}. 
\end{equation}
Here $g_k$ refers to the stochastic subgradient of the objective function $f$ at $\xk$. As demonstrated in \cite{ivgi2023dog}, DoG achieves a near-optimal convergence rate for convex optimization problems and exhibits close performance to the SGD that is equipped with a well-tuned learning rate.  Moreover, \cite{ivgi2023dog} shows that setting the hyperparameter $\rho = 1$ is consistently effective across a broad range of training tasks, particularly in the training of large language and vision models.

For most of the existing learning-rate-free SGD \cite{orabona2014simultaneous,orabona2017training,kempka2019adaptive,chen2022better,carmon2022making,defazio2023learning,ivgi2023dog,mishchenko2023prodigy,kreisler2024accelerated,khaled2024tuning}, despite the fact that empirical studies reveal their great potentials in training neural networks, their theoretical convergence guarantees are only established for convex objective functions or nonconvex differentiable functions. As demonstrated in \cite{bolte2021conservative}, a wide range of modern neural network architectures employ nonsmooth activation functions such as ReLU and leaky ReLU, hence yielding objective functions that lack Clarke regularity \cite{clarke1990optimization} (e.g., differentiability, or weak convexity). To our best knowledge, the theoretical results presented in \cite{orabona2014simultaneous,orabona2017training,kempka2019adaptive,chen2022better,carmon2022making,defazio2023learning,ivgi2023dog,mishchenko2023prodigy} cannot be extended to explain their convergence for training nonsmooth neural networks. 

To develop learning-rate-free stochastic subgradient methods with theoretical guarantees in training nonsmooth neural networks, we employ the recently proposed concept of  {\it conservative field} \cite{bolte2021conservative}. This concept generalizes the Clarke subdifferential and captures the behavior of automatic differentiation algorithms when applied to nonsmooth neural networks. Those functions associated with a conservative field are termed {\it path-differentiable} functions. As discussed in \cite{davis2020stochastic,bolte2021conservative,castera2021inertial}, the class of path-differentiable functions is general enough to enclose those loss functions in a wide range of deep learning tasks.  

Based on the concept of conservative field, some recent works establish the convergence properties for stochastic subgradient methods in minimizing nonsmooth nonconvex path-differentiable functions based on Ordinary Differential Equation (ODE) approaches \cite{clarke1990optimization,benaim2005stochastic,borkar2009stochastic,davis2020stochastic}. These stochastic subgradient methods include SGD \cite{davis2020stochastic,bolte2021conservative,josz2023global}, inertial Newton method (INNA) \cite{castera2021inertial}, momentum SGD \cite{ruszczynski2020convergence,le2023nonsmooth}, Adam-family methods \cite{xiao2024adam}, etc.  In particular, \cite{pauwels2021incremental} proposes a variant of AdaGrad-Norm \cite{ward2020adagrad}, and establishes convergence properties of the proposed method in minimizing nonsmooth path-differentiable functions. In addition, the proof techniques proposed by \cite{pauwels2021incremental} allow for a direct extension to establish the convergence properties for the AdaGrad-Norm \cite{ward2020adagrad} method in the minimization of path-differentiable functions.  However, all the aforementioned stochastic subgradient methods that developed from the ODE approach \cite{davis2020stochastic,ruszczynski2020convergence,bolte2021conservative,castera2021inertial,pauwels2021incremental,hu2022constraint,josz2023global,le2023nonsmooth,xiao2024adam,xiao2023convergence} explicitly involves the learning rate in their update schemes, hence are not learning-rate-free.

As demonstrated in Table \ref{Table_existing_works}, existing learning-rate-free methods are mainly developed for convex optimization problems, leaving a gap in the theoretical understanding of their convergence when applied to neural network training tasks. On the other hand, existing stochastic subgradient methods, developed from ODE approaches for convergence guarantees in nonsmooth nonconvex optimization are not learning-rate-free. In addition, existing works on developing momentum SGD methods for minimizing path-differentiable functions \cite{ruszczynski2020convergence,le2023nonsmooth} require that the momentum parameters converge to $1$. However, practical applications of momentum SGD often employ a constant momentum parameter throughout the iterations. Although \cite{josz2023lyapunov} established the convergence of momentum SGD with constant momentum parameter, their results require the Clarke regularity \cite[Definition 2.3.4]{clarke1990optimization} of the objective function. To the best of our knowledge, the convergence properties of momentum SGD with fixed momentum parameter $\beta$ in minimizing path-differentiable functions remain unexplored.

\begin{table}[t]
    \small
	\centering
	\begin{tabular}{c|cccc}\hline
		& Assumptions on $f$ & Mo. term & Fixed mo. param. & Tuning-free l.r. \\
		\hline
		LFM (Ours) & Path-differentiable & Y & Y &  Y \\\hline
		DoG \cite{ivgi2023dog} & Convex & N  & N &  Y \\
		DoWG \cite{khaled2023dowg} & Convex & N & N & Y \\
		D-adapted GD \cite[Alg. 2]{defazio2023learning} & Convex & N & N & Y \\  
            D-adapted SGD \cite[Alg. 4]{defazio2023learning} & No convergence result & Y & Y & Y \\ 
            U-DOG    \cite[Alg. 1]{kreisler2024accelerated} & Convex & Y & N & Y \\  \hline
		SGD \cite{davis2020stochastic} & Path-differentiable & N & N & N\\
		AdaGrad-Norm \cite{pauwels2021incremental} & Path-differentiable & N & N & N\\
		Momentum SGD \cite{ruszczynski2020convergence,le2023nonsmooth} & Path-differentiable & Y & N & N \\
            INNA \cite{castera2021inertial} & Path-differentiable & Y & N & N \\
		ADAM \cite{xiao2024adam} & Path-differentiable & Y & N & N \\\hline
        Momentum SGD \cite{josz2023lyapunov} & Clarke regular & Y & Y & N \\\hline
	\end{tabular}
\caption{A Comparison of our proposed LFM with existing works on practical learning-rate-free SGD and stochastic subgradient methods. Here ``alg.'', ``mo.'', ``param.'' and ``l.r.'' are the abbreviations for ``algorithm'', ``momentum'', ``parameter'', and ``learning rate'', respectively.}
\label{Table_existing_works}
\end{table}

\paragraph{Motivations}

Motivated by the learning-rate-free SGD schemes developed in DoG method \cite{ivgi2023dog}, we consider a general framework for developing learning-rate-free momentum SGD in the minimization of nonsmooth nonconvex path-differentiable functions,
\begin{equation}
\tag{LFSGD}
\label{Eq_Framework}
    \left\{
    \begin{aligned}
        & \text{Compute $g_k$ as a generalized subgradient of $f_{i_k}$ at $\xk$}, \\
        &\eta_k = \frac{\mu_k}{\sqrt{\varepsilon_0 + \sum_{i = 1}^{k} \tau_i\norm{m_i}^2}},\\
        &\mkp = \beta \mk + (1-\beta) g_k,\\
        &\xkp = \xk - \eta_k \mkp. 
    \end{aligned}
    \right.
\end{equation}
Here the sequence $\{i_k\}$ is chosen from $[N]$ by reshuffling, in the sense that  $\{i_k: jN \leq k< jN+N\} = [N]$ holds for any $j\geq 0$ (i.e., every index in $[N]$ is sampled once in each epoch). Moreover, $\beta \in [0,1)$ is the momentum parameter. Furthermore, we employ the sequences $\{\mu_k\}$ and $\{\tau_k\}$ to characterize how the learning rates are automatically constructed within the framework \eqref{Eq_Framework}. In addition, $\varepsilon_0>0$ is a prefixed constant that enforces the numerical stabilization of the framework \eqref{Eq_Framework}.

When we choose $\mu_k = \max_{i\leq k}\norm{x_i - x_0}$, $\tau_k = 1$ and $\beta = 0$, the framework \eqref{Eq_Framework}  yields DoG \cite{ivgi2023dog}. Alternatively, by choosing $\mu_k = \tau_k = (\max_{i\leq k}\norm{x_i - x_0})^2$ and $\beta = 0$, we recover DoWG developed by \cite{khaled2023dowg}. The flexibility of the framework \eqref{Eq_Framework} enables us to explain the convergence of some existing learning-rate-free methods in training nonsmooth neural networks.

Furthermore, different from existing learning-rate-free methods that are developed from SGD  \cite{ivgi2023dog,khaled2023dowg,defazio2023learning}, our proposed framework \eqref{Eq_Framework} employ the (heavy-ball) momentum term \cite{polyak1964some} for acceleration. In addition, the momentum parameter $\beta$ in \eqref{Eq_Framework} remains constant throughout the iterations, which is identical to the default settings employed in widely-used deep learning frameworks, such as PyTorch and TensorFlow.  To our best knowledge,  there has been no prior research that investigates the convergence properties of \eqref{Eq_Framework} in minimizing nonsmooth path-differentiable functions.

\paragraph{Contributions}

In this paper, we propose a general framework \eqref{Eq_Framework} for developing learning-rate-free SGD with convergence guarantees in nonsmooth nonconvex optimization. In this framework, we draw samples by reshuffling and adaptively compute learning rates $\{\eta_k\}$ by the historical information of the iterates. We prove the convergence properties of the framework \eqref{Eq_Framework}, in the sense that under mild conditions, any cluster point of the sequence $\{\xk\}$ is a stationary point of $f$ and the sequence $\{f(\xk)\}$ converges. Based on these theoretical results, we provide convergence guarantees for DoG and DoWG in training nonsmooth neural networks with random shuffling.  Furthermore, we show that when $f$ is coercive,  under mild assumptions, the sequence $\{\xk\}$ generated by the framework \eqref{Eq_Framework} is uniformly bounded.

We show our proposed framework \eqref{Eq_Framework} covers several efficient learning-rate-free methods, including DoG \cite{ivgi2023dog} and DoWG \cite{khaled2023dowg}. Moreover,  based on the framework \eqref{Eq_Framework}, we propose a novel learning-rate-free momentum-accelerated SGD (LFM) in Algorithm \ref{Alg:PFSGD} for solving \eqref{Prob_Ori}.  To the best of our knowledge, LFM is the first learning-rate-free SGD method that enjoys solid convergence guarantees for training nonsmooth neural networks. We conduct numerical experiments in image classification tasks to evaluate the performance of LFM. The results show that LFM outperforms DoG and Dadap-SGD, and achieves both comparable performance and enhanced robustness as DoWG. These results further demonstrate the promising capabilities of our proposed framework \eqref{Eq_Framework}.

\paragraph{Organization}
The rest of this paper is organized as follows. Section 2 outlines the notation employed throughout and presents the necessary preliminary on nonsmooth analysis and differential inclusion. In Section 3, we analyze the convergence properties for the framework \eqref{Eq_Framework}, and investigate the uniform boundedness of the sequence $\{\xk\}$ under deterministic settings. In Section 4, we develop a learning-rate-free SGD (LFM) in Algorithm \ref{Alg:PFSGD}. In Section 5, we present the results of our numerical experiments that investigate the performance of LFM in training nonsmooth neural networks. Finally, we conclude the paper in the last section.

\section{Preliminaries}
\subsection{Basic notations}
We denote $\inner{\cdot, \cdot}$ as the standard inner product of $\Rn$ and $\norm{\cdot}$ as the $\ell_2$-norm of a vector or an operator, while $\norm{\cdot}_p$ refers to the $\ell_p$-norm of a vector or an operator.  $\bb{B}_{\delta}(x):= \{ \tilde{x} \in \Rn: \norm{\tilde{x} - x}^2 \leq \delta^2 \}$ denotes the ball centered at $x$ with radius $\delta$. For a given set $\ca{Y}$, $\mathrm{dist}(x, \ca{Y})$ denotes the distance between $x$ and  $\ca{Y}$, i.e. $\mathrm{dist}(x, \ca{Y}) := \mathop{\arg\min}_{y \in \ca{Y}} ~\norm{x-y}$,  $\cl \ca{Y}$ denotes the closure of $\ca{Y}$ and $\conv \ca{Y}$ denotes the convex hull of $\ca{Y}$. For any $x \in \Rn$, we denote $\inner{x, \ca{Y}} := \{\inner{x, y}: y \in \ca{Y}\}$.

\subsection{Nonsmooth analysis and conservative field}

In this subsection, we introduce some basic concepts related to nonsmooth analysis, especially to the concept of conservative field. 

\begin{defin}[\cite{clarke1990optimization}]
	\label{Defin_Subdifferential}

  For any given locally Lipschitz continuous function $f: \Rn \to \bb{R}$, the Clarke subdifferential of $f$, denoted by $\partial f$, is defined as 
    \begin{equation*}
        \partial f(x) := \conv\left( \left\{ \lim_{k\to +\infty} \nabla f(\zk): \text{$f$ is differentiable at $\{\zk\}$, and } \lim_{k\to +\infty} \zk = x   \right\}  \right).
    \end{equation*}
\end{defin}

A set-valued mapping $\D: \Rn \rightrightarrows \Rp$ is a mapping that maps  $\Rn$ into the collection of subsets of $\Rp$. For any given set-valued mapping, its graph is given by
\begin{equation}
    \mathrm{graph}(\D) := \{(x, d) \in \Rn \times \Rp: x \in \Rn, d \in \D(x)\}. 
\end{equation}
Moreover, the set-valued mapping $\D$ is said to have closed graph (also called graph-closed) if  $\mathrm{graph}(\D)$ is a closed subset of $\Rn \times \Rp$. In addition, $\D$ is said to be convex-valued, if $\D(x)$ is a convex subset of $\Rp$ for any $x \in \Rn$.  Furthermore, $\D$ is said to be locally bounded if for any compact subset $K$ of $\Rn$, it holds that $\sup_{x \in K, ~d \in \D(x)} \norm{d} <+\infty$.

In the following definitions, we present the definition of conservative field and its corresponding potential function. 
\begin{defin}
	An absolutely continuous curve is an absolutely continuous mapping $\gamma: \bb{R}_+ \to \Rn $ whose derivative $\gamma'$ exists almost everywhere in $\bb{R}$ and $\gamma(t) - \gamma(0)$ equals to the Lebesgue integral of $\gamma'$ between $0$ and $t$ for all $t \in \bb{R}_+$, i.e.,
	\begin{equation*}
		\gamma(t) = \gamma(0) + \int_{0}^t \gamma'(u) \mathrm{d} u, \qquad \text{for all $t \in \bb{R}_+$}.
	\end{equation*}
\end{defin}

\begin{defin}
    \label{Defin_conservative_field}
    Let $f: \Rn \to \bb{R}$ be a locally Lipschitz continuous function, and $\D: \Rn \to \Rn$ be a graph-closed locally bounded set-valued mapping. Then $f$ is a potential function for $\D$ if for any $x \in \Rn$ and any absolutely continuous curve $\gamma$ with $\gamma(0) = 0$ and $\gamma(1) = x$, it holds that 
    \begin{equation}
        f(x) = f(0) + \int_{0}^1 \sup_{l_s \in \D(\gamma(s))} \inner{l_s, \dot{\gamma}(s)} \mathrm{d}s. 
    \end{equation}
    We call the set-valued mapping $\D$ as a conservative field for $f$, or simply a conservative field. Moreover, a function $f$ is called path-differentiable if it is a potential function for some conservative fields.  
\end{defin}

\subsection{Characterizing path-differentiability}

In this subsection, we introduce the results from \cite{davis2020stochastic,bolte2021conservative} on how to characterize the class of potential functions. In particular, we introduce the concept of definable sets and functions. As discussed in \cite{davis2020stochastic,bolte2021conservative}, any definable function is a potential function. More importantly, the class of definable functions is general enough to cover a wide range of nonsmooth loss functions in deep learning.

\subsubsection{Basic introduction to $o$-minimal structures}
As illustrated in \cite{pauwels2023conservative}, sets or functions are called definable if they can be characterized by a finite number of basic formulas, inequalities, or Boolean operations involving standard functions such as polynomial, exponential, or coordinate-wise maximum. For a more comprehensive introduction of definable functions and sets, interested readers could refer to \cite{bolte2021conservative, davis2020stochastic} for detailed explanations and illustrations.

The concept of definability is closely related to the concept of the $o$-minimal structure. 
\begin{defin}
An $o$-minimal structure on $(\bb{R}, +, \cdot)$ is a countable collection of sets $\ca{O} = \{\ca{O}_p: p \in \bb{N}\}$, where each $\ca{O}_p$ is itself a family of subsets of $\Rp$, such that 
\begin{enumerate}
    \item $\ca{O}_p$ contains the empty set and is stable by finite union, finite intersection, and complementation.
    \item If $A$ belongs to $\ca{O}_p$, then both $A \times \bb{R}$ and $\bb{R} \times A$ belong to $\ca{O}_{p+1}$.
    \item If $\pi:\Rp \times \bb{R} \to \Rp$ denotes the coordinate projection onto $\Rp$, then for any $A$ in $\ca{O}_{p+1}$, $\pi(A)$ belongs to $\ca{O}_p$.
    \item  $\ca{O}_p$ contains all the sets of the form $\{x \in \Rp: \nu(x) = 0\}$,  where $\nu$ is any polynomial on $\Rp$. 
    \item The elements in $\ca{O}_1$ are exactly the finite unions of intervals and points.
\end{enumerate}
Moreover, a mapping $F: \Rn \to \Rm$ is called definable in $\ca{O}$ if its graph is definable in $\ca{O}$ as a subset of $\Rn \times \Rm$, i.e.,  ${\rm graph}(F) \in \ca{O}_p$ for some $p$.

\end{defin}

The terminology "tameness" refers to definability in an $o$-minimal structure without specifying its structure. In many cases, the $o$-minimal structure is predetermined, allowing us to use the term "definable" instead of "tame" throughout this paper. The following proposition demonstrates that any definable function can serve as a potential function for a specific conservative field. 
\begin{prop}[Theorem 5.8 in \cite{davis2020stochastic}]
    Suppose $f: \Rn \to \bb{R}$ is a locally Lipschitz continuous definable function. Then $f$ is a potential function that admits $\partial f$ as its conservative field.  
\end{prop}

Proposition \ref{Prop_Definable_operations} illustrates that definable sets exhibit the same robustness properties and attractive analytic features as semi-algebraic sets. 
\begin{prop}[Proposition 1 in \cite{pauwels2023conservative}]
    \label{Prop_Definable_operations}
    \begin{enumerate}
        \item The closure, interior, and boundary of a definable set are definable.
        \item Images and inverse images of definable sets under definable maps are definable.
        \item Compositions of definable maps are definable.
    \end{enumerate}
\end{prop}

It is worth mentioning that the class of definable functions is general enough to encompass a wide range of objective functions prevalent in practical machine-learning scenarios.  As illustrated in the Tarski-Seidenberg theorem \cite{bierstone1988semianalytic}, any semi-algebraic function is definable. Moreover, \cite{wilkie1996model} demonstrates the existence of an $o$-minimal structure that contains both the graph of the exponential function and all semi-algebraic sets. Consequently, numerous commonly employed activation and loss functions, such as sigmoid, softplus, ReLU, leaky ReLU, $\ell_1$-loss, MSE loss, hinge loss, logistic loss, and cross-entropy loss, fall within the class of definable functions. Furthermore, Proposition \ref{Prop_Definable_operations} illustrates that definability is preserved under finite summation and composition.  Consequently, any neural network built from definable building blocks has a definable loss function. More importantly, the outputs of the AD algorithms can be described by a definable conservative field. This fact leads to the following proposition, which exhibits the nonsmooth Morse-Sard property for \eqref{Prob_Ori}. 
    \begin{prop}
        \label{Prop_definable_regularity}
        Let $f$ be a potential function that admits $\D_f$ as its conservative field. Suppose $f$ and $\D_f$ are definable over $\Rn$, then the set $\{f(x):   0 \in \conv(\D_f(x) )\}$ is a finite subset of $\bb{R}$. 
    \end{prop}
    \begin{proof}
        According to \cite[Theorem 4]{bolte2021conservative}, $(f, D_f)$ has a $C^r$ variational stratification \cite[Definition 5]{bolte2021conservative}. That is, there exists a stratification $\{\M_i: i \in I\}$ of $\Rn$ such that $\mathrm{Proj}_{T_x\M_x}\left( D_f(x)  \right) = \{\grad f(x)\}$. 
        Here  $\grad f$ refers to the Riemannian gradient of $f$ restricted to the active strata $M_x$ containing $x$. In addition, $T_x\M_x$ refers to the tangent space of $M_x$ at $x$.  It follows from the linearity of the mapping $\mathrm{Proj}_{T_x\M_x}$ that 
        \begin{equation}
            \mathrm{Proj}_{T_x\M_x}\left( \conv( D_f(x))  \right) = \conv\left( \mathrm{Proj}_{T_x\M_x}\left( D_f(x)  \right)\right) = \{\grad f(x)\}. 
        \end{equation}
        As a result, $(f, x\mapsto  \conv(D_f(x)) )$ also has a $C^r$ variational stratification. Then it follows from \cite[Corollary 5]{bolte2021conservative}, by the same proof techniques as \cite[Corollary 5]{bolte2007clarke} and \cite[Theorem 5]{bolte2021conservative} (i.e., applying the definable Sard theorem to each stratum), we can conclude that set $\{f(x):   0 \in \conv(\D_f(x) )\}$ is a finite subset of $\bb{R}$. This completes the proof. 
    \end{proof}

    \subsection{Differential inclusion and ODE approaches for subgradient methods}
    In this subsection, we introduce some fundamental concepts related to the stochastic approximation technique that are essential for the proofs presented in this paper. The concepts discussed in this subsection are mainly from \cite{benaim2005stochastic}. Interested readers could refer to \cite{benaim2006dynamics,benaim2005stochastic,borkar2009stochastic,davis2020stochastic} for more details on the stochastic approximation technique. 

        We first present the definition for the trajectory of a differential 
        inclusion. 
	\begin{defin}
		For any locally bounded graph-closed set-valued mapping $\D: \Rn \rightrightarrows \Rn$,  we say that an absolutely continuous path $x(t)$ in $\Rn$ is a solution for the differential inclusion 
		\begin{equation}
			\label{Eq_def_DI}
			\frac{\mathrm{d} x}{\mathrm{d}t} \in \D(x),
		\end{equation}
		with initial point $x_0$ if $x(0) = x_0$, and $\dot{x}(t) \in \D(x(t))$ holds for almost every $t\geq 0$. 
	\end{defin}
        It is worth mentioning that the results in \cite{aubin2012differential} guarantee the existence of the trajectories of \eqref{Eq_def_DI}.  In the following, we present the definition of the expansion for a given set-valued mapping. 
	\begin{defin}
		\label{Defin_delta_expansion}
		For any given set-valued mapping $\D: \Rn \rightrightarrows \Rn$ and any constant $\delta \geq 0$,  the $\delta$-expansion of $\D$, denoted as $\D^{\delta}$,  is defined as
		\begin{equation}
			\D^{\delta}(x) := \{ w \in \Rn:  \, \exists\, z \in \bb{B}_{\delta}(x) 
   \;\mbox{such that} \; \mathrm{dist}(w, \ca{D}(z))\leq \delta\}.
		\end{equation}
	\end{defin}

        Now consider the sequence $\{\xk\}$ generated by the  following update scheme,  
	\begin{equation}
		\label{Eq_def_Iter}
		\xkp = \xk + \eta_k d_k ,
	\end{equation}
        with a diminishing positive sequence of real numbers $\{\eta_k\}$ as the learning rates. Moreover, $d_k$ represents an inexact evaluation of $\D$ at $\xk$. To illustrate the relationship between the discrete sequence $\{\xk\}$ and the trajectories of the differential inclusion \eqref{Eq_def_DI}, we first introduce the concepts of {\it interpolated process} and {\it perturbed solution} in the following two definitions. 
	\begin{defin}
            \label{Defin_interpolated_process}
		For any given sequence of positive real numbers $\{\eta_k\}$, the  (continuous-time) interpolated process of $\{\xk\}$ with respect to $\{\eta_k\}$ is the mapping $w: \bb{R}_+ \to \Rn$ such that 
		\begin{equation}
			w(\lambda_i + s) := x_{i} + \frac{s}{\eta_i} \left( x_{i+1} - x_{i} \right), \quad s\in[0, \eta_i). 
		\end{equation}
		Here $\lambda_0 := 0$, and $\lambda_i := \sum_{k = 0}^{i-1} \eta_k$ for $i\geq 1$.
	\end{defin}

        \begin{defin}
		\label{Defin_perturbed_solution}
		We say an absolutely continuous function $\gamma$
		is a perturbed solution to \eqref{Eq_def_DI}  if there exists a locally integrable function $u: \bb{R}_+ \to \Rn$, along with $\delta_f: \bb{R}_+ \to \bb{R}$, such that $\lim\limits_{t \to \infty} \delta_f(t) = 0$ and $\dot{\gamma}(t)  \in \D^{\delta_f(t)}(\gamma(t))$. 
	\end{defin}

    The following lemma is an extension of \cite[Proposition 1.3]{benaim2005stochastic}, which allows for inexact evaluations of the set-valued mapping $\D$. It shows that for the sequence $\{\xk\}$ generated by \eqref{Eq_def_Iter}, its interpolated process with respect to $\{\eta_k\}$ is a perturbed solution of the differential inclusion \eqref{Eq_def_DI}.   
        \begin{lem}
		\label{Le_interpolated_process}
		Let $\ca{D}: \Rn \rightrightarrows \Rn$ be a locally bounded set-valued mapping that is nonempty compact convex valued and graph-closed.
		Suppose the following conditions hold in \eqref{Eq_def_Iter}:
		\begin{enumerate}
			\item $\lim_{k\to +\infty} \eta_k = 0$, $\sum_{k = 0}^{+\infty} \eta_k = +\infty$.  
			\item There exist a nonnegative sequence $\{\delta_k\}$  such that $\lim_{k\to \infty} \delta_k = 0$ and $d_k \in \D^{\delta_k}(\xk)$.
			\item $\sup_{k \geq 0} \norm{\xk}<\infty$, $\sup_{k \geq 0} \norm{d_k} < \infty$. 
		\end{enumerate}
		Then the interpolated process of $\{\xk\}$ is a perturbed solution for \eqref{Eq_def_DI}. 
	\end{lem}

\section{Convergence Analysis}

In this section, we present the convergence properties of the framework \eqref{Eq_Framework}. Section 3.1 outlines the basic assumptions and the main result in Section 3.1, while Section 3.2 presents the detailed proofs. Furthermore, Section 3.3 exhibits the global stability of the framework \eqref{Eq_Framework}, in the sense that the sequence $\{\xk\}$ remains uniformly bounded with sufficiently small stepsizes. 

\subsection{Basic assumptions and main result}

To establish the convergence properties of \eqref{Eq_Framework}, we first make the following assumptions on the objective function $f$ in \eqref{Prob_Ori}. 
\begin{assumpt}
    \label{Assumption_f}
    \begin{enumerate}
        \item For each $i \in [N]$,  $f_i$ is a locally Lipschitz continuous path-differentiable function that admits $\D_{f_i}$ as its conservative field. 
        \item The set $\left\{f(x):  0 \in \conv\left( \frac{1}{N} \sum_{i = 1}^N \D_{f_i}(x) \right) \right\}$ has empty interior in $\bb{R}$. 
        \item $f$ is bounded from below over $\Rn$. 
    \end{enumerate}
\end{assumpt}

As discussed in \cite{bolte2021conservative} and Section 2.2, the class of path-differentiable functions is general enough to enclose a wide range of optimization problems, particularly in training nonsmooth neural networks. Therefore, Assumption \ref{Assumption_f}(1) is mild in practical implementations. 

Based on the chain rule for conservative fields demonstrated in \cite{bolte2021conservative,bolte2021nonsmooth}, the set-valued mapping $x \mapsto \frac{1}{N} \sum_{i = 1}^N \D_{f_i}(x)$ is a conservative field for $f$. Therefore, throughout this paper, we choose the conservative field for $f$ as 
\begin{equation}
   \D_f(x) := \conv\left(\frac{1}{N} \sum_{i = 1}^N \D_{f_i}(x) \right). 
\end{equation}

Moreover, as demonstrated in Proposition \ref{Prop_definable_regularity}, Assumption \ref{Assumption_f}(2) holds whenever $\{f_i:i\in [N]\}$ and $\{\D_{f_i}:i\in [N]\}$ are definable. Therefore, Assumption \ref{Assumption_f}(2) is a mild assumption in practice.

Furthermore, following the assumptions in \cite{davis2020stochastic,castera2021inertial,le2023nonsmooth,xiao2024adam}, we introduce the assumptions on the framework \eqref{Eq_Framework} as described below.
\begin{assumpt}
    \label{Assumption_Framework}
    \begin{enumerate}
        \item The sequence of indexes $\{i_k\}$ is generated by reshuffling. That is, for any $j\geq 0$, it holds that $\{i_k: jN\leq k< (j+1)N\} = [N]$. 
        \item The sequence of iterates $\{\xk\}$ is uniformly bounded.
        \item The sequences $\{\tau_k\}$ and $\{\mu_k\}$ are positive, non-decreasing. Moreover, there exists a locally bounded function $\psi: \Rn \to \bb{R}_+$ such that $\max\{\tau_k , \mu_k \} \leq \sup_{0\leq i\leq k}\psi(x_i)$ holds for any $k\geq 0$. 
        \item There exists a diminishing sequence $\{\delta_k\}$ such that $g_k \in \D^{\delta_k}_{f_{i_k}}(\xk)$.
    \end{enumerate}
\end{assumpt}

Here are some comments on Assumption \ref{Assumption_Framework}. Assumption \ref{Assumption_Framework}(1) illustrates how the indexes $\{i_k\}$ are generated from $[N]$ in each epoch. In particular, as we do not assume any randomness in choosing $i_k$, Assumption \ref{Assumption_Framework}(1) encompasses a wide range of 
sampling methods for selecting the indexes from $[N]$, including sequential selection and random reshuffling. 
In addition, 
as Lemma \ref{Le_close_xk} illustrates that $\lim_{k \to +\infty} \norm{\xkp - \xk} = 0$, Assumption \ref{Assumption_Framework}(1) covers the scenarios where the indexes $\{i_k\}$ are chosen by mini-batch sequential selection and mini-batch random reshuffling.

Moreover, Assumption \ref{Assumption_Framework}(2) assumes the uniform boundedness of $\{\xk\}$, which is a common assumption in various existing works. Assumption \ref{Assumption_Framework}(3) is a mild condition that can be satisfied in various existing learning-rate-free methods, such as DoG and DoWG. It is easy to verify that the uniform boundedness of $\{\xk\}$ implies the uniform boundedness of the sequences $\{\mu_k\}$ and $\{\tau_k\}$.  Furthermore, Assumption \ref{Assumption_Framework}(4) characterizes how the gradient approximation $g_k$ correlates with the subdifferential $\D_{f_{i_k}}(\xk)$ at the $k$-th iteration.

In the following theorem, we prove the global convergence properties of the framework \eqref{Eq_Framework}. The detailed proof of Theorem \ref{The_convergence} is presented in Section 3.2 for simplicity. 
\begin{theo}
    \label{The_convergence}
     Let $\{\xk\}$ be the sequence generated by \eqref{Eq_Framework}. Suppose Assumption \ref{Assumption_f} and Assumption \ref{Assumption_Framework} hold. Then  any cluster point of sequence $\{\xk\}$ lies in $\{x \in \Rn: 0\in \D_f(x) \}$, and the sequence $\{f(\xk)\}$ converges. 
\end{theo}

\subsection{Proof for Theorem \ref{The_convergence}}

In this subsection, we present a detailed proof for Theorem \ref{The_convergence}. We begin our proof with the following lemma illustrating that the sequences $\{g_k\}$, $\{\mk\}$ and $\{\eta_k\}$ are uniformly bounded.  The proof for Lemma 1 follows directly from the uniform boundedness of $\{\xk\}$ stated in Assumption \ref{Assumption_Framework}(2), along with the fact that $\D_{f_i}$ is both graph-closed and locally bounded for any $i \in [N]$, as shown in Assumption \ref{Assumption_f}(1). Therefore, we omit the proof of Lemma \ref{Le_UB_gk} for the sake of simplicity. 
\begin{lem}
    \label{Le_UB_gk}
    Suppose Assumption \ref{Assumption_f} and Assumption \ref{Assumption_Framework}(2)-(4) hold, then for the sequence   $\{g_k\}$ generated by \eqref{Eq_Framework}, we have that $\sup_{k\geq 0} \norm{g_k} + \norm{\mk} <+\infty$ and $\sup_{k\geq 0} \eta_k < +\infty$. 
\end{lem}

It is worth mentioning that the sequence of learning rates $\{\eta_k\}$ in the framework \eqref{Eq_Framework} can be non-diminishing. In the following lemma, we present an equivalent condition for the scenarios where sequence $\{\eta_k\}$ is non-diminishing. 
\begin{lem}
    \label{Le_stepsizes}
    Suppose Assumption \ref{Assumption_Framework}(2)-(4) holds. Then $\mathop{\lim\inf}_{k\to +\infty} ~\eta_k >0$ if and only if $\sum_{k = 1}^{+\infty} \tau_k\norm{\mk}^2 <+\infty$. 
\end{lem}
\begin{proof}
    When $\sum_{k = 1}^{+\infty} \tau_k\norm{\mk}^2 <+\infty$, it directly follows from Assumption \ref{Assumption_Framework}(3) 
    that 
    \begin{equation}
        \mathop{\lim\inf}_{k\to +\infty} ~\eta_k \geq \lim_{k\to +\infty} \frac{\inf_{k\geq 0}\mu_k}{ \sqrt{\varepsilon_0 +\sum_{i = 1}^{k} \tau_i \norm{m_i}^2 }} \geq \lim_{k\to +\infty} \frac{\mu_0}{ \sqrt{\varepsilon_0 +\sum_{i = 1}^{k} \tau_i \norm{m_i}^2 }} > 0. 
    \end{equation}
    Conversely, when $\mathop{\lim\inf}_{k\to +\infty} ~\eta_k >0$, it holds that 
    \begin{equation}
        0< \mathop{\lim\inf}_{k\to +\infty} ~\eta_k \leq \lim_{k\to +\infty} \frac{\sup_{i\geq 0}\mu_i }{\sqrt{\varepsilon_0 +\sum_{i = 1}^{k} \tau_i \norm{m_i}^2}}.
    \end{equation}
    From Assumption \ref{Assumption_Framework}(2), we can conclude that $\sup_{k\geq 0} \mu_k <+\infty$ under Assumption 3.2(3). Hence $\lim_{k\to +\infty}\sum_{i = 1}^{k} \tau_i \norm{m_i}^2 <+\infty$. This completes the proof. 
\end{proof}

Based on Lemma \ref{Le_stepsizes}, we present following lemma to show that $\lim_{k\to +\infty}\norm{\xkp - \xk} = 0$ in \eqref{Eq_Framework}. 
\begin{lem}
    \label{Le_close_xk}
    Suppose Assumption \ref{Assumption_f} and Assumption \ref{Assumption_Framework}(2)-(4) hold, then it holds that 
    \begin{equation}
        \lim_{k\to +\infty} \sup_{k\leq j\leq k+N} \norm{x_j - \xk}  = 0. 
    \end{equation}
\end{lem}
\begin{proof}
    When  $\mathop{\lim\inf}_{k\to +\infty} \eta_k > 0$, we have that $\sum_{k = 1}^{+\infty} \tau_k \norm{\mk}^2 < +\infty$, and hence $\lim_{k\to +\infty} \norm{\mk} = 0$. In addition, 
    by  Lemma \ref{Le_UB_gk},
    it holds that $\sup_{k\geq 0} \eta_k < +\infty$.
    Thus, we can conclude that 
    \begin{equation}
        \lim_{k\to +\infty} \norm{\xkp - \xk} \leq \mathop{\lim\sup}_{k\to +\infty} \eta_k \lim_{k\to +\infty} \norm{\mk} = 0. 
    \end{equation}

    On the other hand, when $\mathop{\lim}_{k\to +\infty} \eta_k = 0$, from the fact that sequence $\{\mk\}$ is uniformly bounded in Lemma \ref{Le_UB_gk}, we can conclude that $\lim_{k\to +\infty} \norm{\xkp - \xk} = 0$. This completes the proof. 
\end{proof}

In the following, we present Lemma \ref{Le_approximate_evaluation} as an auxiliary lemma that is essential in characterizing that the sequence $\{\mk\}$ can be regarded as an approximated evaluation for $\D_f(\xk)$.

      \begin{lem}
            \label{Le_approximate_evaluation}
            Suppose Assumption \ref{Assumption_f} holds. Given any diminishing positive sequence $\{\delta_k\}$,  any uniformly bounded points $\{\xk\}$, and any sequence $\{d_k\}$ that satisfies $d_k \in \conv\big( \D_f^{\delta_k}(\xk) \big)$. 
            Then for any constant $T>0$, there exists $K_T>0$ such that $d_k \in \D_f^T(\xk)$ holds for any $k\geq K_T$. Moreover, there exists a diminishing positive sequence $\{\tilde{\delta}_k\}$ such that $d_k \in \D_f^{\tilde{\delta}_k}(\xk)$ holds for any $k\geq 0$.
        \end{lem}
        \begin{proof}
            We prove this lemma by contradiction. Suppose there exists a  sequence $\{\delta_k\}$, a uniformly bounded sequence $\{\xk\}$,  a sequence $\{d_k\}$ that satisfies $d_k \in \conv\big( \D_f^{\delta_k}(\xk) \big)$, and a constant $T>0$ such that $d_k \notin \D_f^{T}(\xk)$ for infinitely many $k\geq 0$. 

            Then by the uniformly boundedness of $\{\xk\}$,  we can choose a sequence of indexes $\{k_j\}$ satisfying  $x_{k_j} \to \bar{x}$, and $ d_{k_j}\notin \D_f^{T}(x_{k_j})$ for any $j\geq 0$. Notice that $\bar{x} \in \bb{B}_{T}(\xk)$ holds for any sufficiently large $k$, hence $\D_f(\bar{x}) + \bb{B}_{T}(0) \subseteq \D_f^{T}(\xk)$ holds for any sufficiently large $k$. Together with the fact that $ d_{k_j}\notin \D_f^{T}(x_{k_j})$, we can conclude that 
            \begin{equation}
                \label{Eq_Le_approximate_evaluation_0}
                \mathop{\lim\inf}_{j\to +\infty}~\mathrm{dist}(d_{k_j}, \D_f(\bar{x}) )  \geq T >0.
            \end{equation}

            On the other hand, from the fact that $\D_f$ is a convex-valued graph-closed mapping, it holds that $\cap_{\delta>0} \D_f^{\delta}(\bar{x}) = \D_f(\bar{x})$. Then from the fact that $d_k \in \conv\big( \D_f^{\delta_k}(\xk) \big)$, we get
            \begin{equation}
                \label{Eq_Le_approximate_evaluation_1}
                \lim_{j\to+\infty} \mathrm{dist}\left( d_{k_j}, \D_f(\bar{x}) \right) = 0. 
            \end{equation}
            As a result, we obtain a contradiction between \eqref{Eq_Le_approximate_evaluation_0} and \eqref{Eq_Le_approximate_evaluation_1}, and thereby our assumption is not valid.  Thus we can conclude that for any constant $T>0$, there exists $K_T>0$ such that $d_k \in \D_f^{T}(\xk) $ for any $k\geq K_T$. This completes the first part of the proof.

            Finally, by choosing $T$ from $\{\frac{1}{i}:i=1,2,...\}$, and setting $\tilde{\delta}_k = \frac{1}{i}$ for any $k \in \left[K_{\frac{1}{i}}, K_{\frac{1}{i+1}}\right]$, we have that $d_k \in \D_f^{\tilde{\delta}_k}(\xk)$ holds for any $k\geq K_1$. This completes the entire proof. 
            
        \end{proof}

In the next proposition, we prove the convergence properties of \eqref{Eq_Framework} under the scenarios where the sequence of learning rates $\{\eta_k\}$ is non-diminishing. 
\begin{prop}
    \label{Prop_case1_etak_bounded}
    Suppose Assumption \ref{Assumption_f} and Assumption \ref{Assumption_Framework} hold, and $\mathop{\lim\inf}_{k\to +\infty} \eta_k > 0$. Then any cluster point of the sequence $\{\xk\}$ is a $\D_f$-stationary point of $f$, and the sequence of function values $\{f(\xk)\}$ converges.  
\end{prop}
\begin{proof}
    As discussed in Lemma \ref{Le_close_xk}, when  $\mathop{\lim\inf}_{k\to +\infty} \eta_k > 0$, we have that $\lim_{k\to +\infty} \norm{\mk} = 0$ and $\lim_{k\to +\infty} \norm{\xkp - \xk} = 0$. 
    Then let 
    \begin{equation}
        \hat{\delta}_k := \sup_{k\leq j< k+N}\norm{x_{j} - \xk} + \sum_{l = k}^{k+N} \delta_l.
    \end{equation}
    It is easy to verify that $\lim_{k\to +\infty} \hat{\delta}_k = 0$.  Moreover,  from the update scheme of the sequence $\{\mk\}$, we have that 
    \begin{equation}
        \lim_{k\to +\infty} \norm{g_k} \leq \lim_{k\to +\infty} \frac{1}{1-\beta}\left( \norm{\mkp} + \beta \norm{\mk} \right) = 0. 
    \end{equation}
    Thus it holds that 
    \begin{equation}
        \label{Eq_Prop_case1_etak_bounded_1}
        \lim_{j\to +\infty} \norm{\frac{1}{N}\sum_{l = jN}^{(j+1)N-1} g_l  } = 0. 
    \end{equation}

    Notice that $g_l \in \D_{f_{i_l}}^{\delta_{jN}}(x_{jN})$ holds for any $j \geq 0$ and any $l\in \{jN,\ldots, (j+1)N-1\}$, we can conclude that 
    \begin{equation}
        \frac{1}{N}\sum_{l = jN}^{(j+1)N-1} g_l \in \frac{1}{N} \sum_{l = jN}^{(j+1)N-1}  \D_{f_{i_l}}^{\delta_{jN}}(x_{jN}) \subseteq \conv\left( \D_{f}^{\hat{\delta}_{jN}}(x_{jN}) \right). 
    \end{equation}
    As a result, from \eqref{Eq_Prop_case1_etak_bounded_1} and Lemma \ref{Le_approximate_evaluation}, there exists a non-negative diminishing sequence $\{\tilde{\delta}_k\}$ such that 
    \begin{equation}
        \frac{1}{N}\sum_{l = jN}^{(j+1)N-1} g_l \in \D_f^{\tilde{\delta}_{jN}}(x_{jN}), \quad \forall j\geq 0. 
    \end{equation}
    Together with \eqref{Eq_Prop_case1_etak_bounded_1}, we arrive at $\lim_{j \to +\infty} \mathrm{dist}\left(0, \D_f(x_{jN})  \right) = 0$. Then together with Lemma \ref{Le_close_xk} and the fact that $\D_f$ is graph-closed, we can conclude that 
    \begin{equation}
        \lim_{k \to +\infty} \mathrm{dist}\left(0, \D_f(\xk)  \right) = 0.
    \end{equation}
    Therefore, any cluster point of the sequence $\{\xk\}$ lies in the subset $\{x \in \Rn: 0 \in \D_f(x)\}$. 

    Furthermore, let $\ca{B}$ be the set of all cluster points of the sequence $\{\xk\}$. That is,   $\ca{B} = \{x \in \Rn:  \text{ there exists an increasing sequence of integers  $\{k_j\}$ such that } \lim_{j\to +\infty} x_{k_j} = x\} $. Then, from the definition of $\ca{B}$, we can conclude that $\{f(x): x \in \ca{B}\}$ is a closed subset of $\bb{R}$. 
    
    Next we claim that $\{f(x): x \in \ca{B}\}$ is a connected subset of $\bb{R}$. Otherwise, let $\{k_{j,1}\}$ and $\{k_{j,2}\}$ be two increasing subsequences of integers such that 
    \begin{equation}
        \lim_{j\to +\infty} x_{k_{j,1}} = x_1, \quad \lim_{j\to +\infty} x_{k_{j,2}} = x_2, \quad q_1:=\lim_{j\to +\infty} f(x_{k_{j,1}}) <  q_2:=\lim_{j\to +\infty} f(x_{k_{j,2}}). 
    \end{equation}
    Notice that $\lim_{k\to +\infty} \norm{\xkp - \xk} = 0$ implies  $\lim_{k\to +\infty} |f(\xkp) - f(\xk)| = 0$. Then $\{f(x): x \in \ca{B}\} \cap [q_1, q_2]$ must be dense and thus connected in $\bb{R}$. Therefore, $\{f(x): x \in \ca{B}\}$ is a connected subset of $\bb{R}$. 

    Together with  Assumption \ref{Assumption_f}(2) stating that $\{f(x):  0\in \D_f(x)\}$ has empty interior in $\bb{R}$, we can conclude that $\{f(x): x \in \ca{B}\}$ is a singleton. This completes the proof. 
\end{proof}

In the rest of this subsection, we aim to prove the convergence properties of \eqref{Eq_Framework} under the scenarios where the sequence of learning rates $\{\eta_k\}$ is non-diminishing. The following lemma demonstrates that the sequence $\{\eta_k\}$ diminishes at a sublinear rate. 
\begin{lem}
    \label{Le_stpesize_ratio}
    Suppose Assumption \ref{Assumption_f} and Assumption \ref{Assumption_Framework} hold, and $\lim_{k\to +\infty} \eta_k = 0$. Then, it holds that 
    \begin{equation}
        \lim_{k\to +\infty} \frac{\eta_{k+1}}{\eta_k} = 1. 
    \end{equation}
\end{lem}
\begin{proof}
    From Lemma \ref{Le_stepsizes}, when $\lim_{k\to +\infty} \eta_k = 0$, it holds that $\sum_{k = 0}^{+\infty} \tau_k \norm{\mk}^2 =+\infty$. 
    Together with the uniform boundedness of the sequence $\{\mk\}$ in Lemma \ref{Le_UB_gk}, we can conclude that 
    \begin{equation}
        \label{Eq_Le_stpesize_ratio_0}
        \lim_{k\to +\infty} \frac{\varepsilon_0 + \sum_{j = 0}^{k+1} \tau_j \norm{m_j}^2}{\varepsilon_0 + \sum_{j = 0}^k \tau_j \norm{m_j}^2} = 1 + \lim_{k\to +\infty} \frac{\tau_{k+1} \norm{\mkp}^2}{\varepsilon_0 + \sum_{j = 0}^k \tau_j\norm{m_j}^2} = 1. 
    \end{equation}
    Furthermore, from Assumption \ref{Assumption_Framework}(2), the sequence $\{\mu_k\}$ is uniformly bounded and non-decreasing. Therefore, it is easy to verify that 
    \begin{equation}
        \label{Eq_Le_stpesize_ratio_1}
        \lim_{k\to +\infty} \frac{\mu_{k+1}}{\mu_k} = 1. 
    \end{equation}
    Therefore, combining \eqref{Eq_Le_stpesize_ratio_0} and 
    \eqref{Eq_Le_stpesize_ratio_1}, we can conclude that 
    \begin{equation}
        \begin{aligned}
            &\lim_{k\to +\infty} \frac{\eta_{k+1}}{\eta_k} 
            = \lim_{k\to +\infty} \sqrt{\frac{\varepsilon_0 + \sum_{j = 0}^{k} \tau_j \norm{m_j}^2}{\varepsilon_0 + \sum_{j = 0}^{k+1} \tau_j \norm{m_j}^2}} \cdot \frac{\mu_{k+1}}{\mu_k} = 1. 
        \end{aligned}
    \end{equation}
    This completes the proof. 
\end{proof}

As shown in Lemma \ref{Le_close_xk}, Lemma \ref{Le_stpesize_ratio}, when $ \lim_{k\to +\infty}\eta_k = 0$, for any positive integer $T>0$, it holds that 
\begin{equation}
    \label{Eq_accumulation_T}
    \lim_{k\to +\infty}\sup_{ k - TN\leq i, l\leq k} \norm{x_i - x_l} = 0, \quad \lim_{k\to +\infty}  \sum_{ k - TN \leq i\leq k} \eta_i = 0, \quad \lim_{k\to +\infty} \sup_{ k - TN\leq i, l\leq k} \left|\frac{\eta_i}{\eta_l} - 1\right| = 0.
\end{equation}
Therefore, we have the following auxiliary lemma. 
\begin{lem}
    \label{Le_Defin_tk}
    Suppose Assumption \ref{Assumption_f} and Assumption \ref{Assumption_Framework} hold, and $\lim_{k\to +\infty} \eta_k = 0$. Then there exists an increasing sequence $\{t_j\} \subset  \{jN: j\geq 0\}$ such that 
\begin{equation}
    \label{Eq_Le_Defin_tk_0}
    \begin{aligned}
        &\lim_{j\to +\infty} t_{j+1}-t_{j} = +\infty, \quad \lim_{j\to +\infty}  
        \sum_{ {t_{j}- T_j} \leq i\leq t_{j+1}} \eta_i = 0, \quad   \lim_{j\to +\infty} \sup_{ {t_{j} - T_j}\leq i, l\leq t_{j+1}} \left|\frac{\eta_i}{\eta_l} - 1\right| = 0,  
    \end{aligned}
\end{equation}
where $T_j := t_{j+1} - t_j$. 
\end{lem}
\begin{proof}
    From \eqref{Eq_accumulation_T}, we can conclude that for any $\hat{K} > 0$, there exists $k_{\hat{K}}$ such that for any $k\geq k_{\hat{K}}$, it holds that
    \begin{equation}
        \label{Eq_Le_Defin_tk_1}
        \sup_{ k - \hat{K}N\leq i, l\leq k} \norm{x_i - x_l} \leq \frac{1}{\hat{K}}, \quad \sum_{ k - \hat{K}N \leq i\leq k + \hat{K}N} \eta_i \leq \frac{1}{\hat{K}}, \quad  \sup_{ k - \hat{K}N\leq i, l\leq k+ \hat{K}N} \left|\frac{\eta_i}{\eta_l} - 1\right| \leq \frac{1}{\hat{K}}.
    \end{equation}
    Then by choosing $\hat{K} = 1,2,\ldots $, we can choose $k_1\leq k_2\leq \ldots$ such that for any $i \in \bb{N}_+$, \eqref{Eq_Le_Defin_tk_1} holds  with $\hat{K} = i$ for any $k\geq k_i$. Then, we can recursively define $k^*_{i+1} = \lceil(k_{i+1}-k^*_i)/i\rceil * i + k^*_i$ with $k^*_1 = 0$. Furthermore, we recursively define $t_j$ by 
    $t_0 = 0$ and $t_{j+1} = t_j + i_j N$, where $i_j = \sup \{i: t_j \geq k^*_i\}$. Then, from the definition of $\{t_j\}$, we can easily verify that \eqref{Eq_Le_Defin_tk_0} holds. This completes the proof.     
\end{proof}

The following lemma illustrates that the average of the sequence $\{g_k\}$ over $t_j \leq k \leq t_{j+1}$ is an approximated evaluation to $\D_f(x_{t_j})$. 
\begin{lem}
    \label{Le_aggregrate_gk}
    Suppose Assumption \ref{Assumption_f} and Assumption \ref{Assumption_Framework} hold, and $\lim_{k\to +\infty} \eta_k = 0$. Then for any $j\geq 0$, there exists a non-negative diminishing sequence $\{\tilde{\delta}^{\star}_k\}$ such that   
    \begin{equation}
        \frac{1}{T_j}\sum_{k = t_j}^{t_{j+1}-1} g_{k-l} \in   \D_f^{\tilde{\delta}^{\star}_{t_j}}(x_{t_{j}}),
    \end{equation}
    holds for any $k\geq 0$ and any  $0 \leq l \leq T_j$. Here the sequence $\{t_j\}$ are chosen as in \eqref{Eq_Le_Defin_tk_0}. 
\end{lem}
\begin{proof}
    From Assumption \ref{Assumption_Framework}(2) and Lemma \ref{Le_Defin_tk}, there exists a diminishing nonnegative sequence $\{\delta^{\star}_k\}$ such that for any $l\leq T_j$, $g_{k-l} \in \D_{f_{i_{k-l}}}^{\delta_{k-l}^{\star}}(x_{t_j})$. 
    As a result, it holds that 
    \begin{equation}
        \frac{1}{T_j}\sum_{k = t_j}^{t_{j+1}-1} g_{k-l} \in \frac{1}{T_j}\sum_{k = t_j}^{t_{j+1}-1}\D_{f_{i_{k-l}}}^{\delta_{k-l}^{\star}}(x_{t_j}). 
    \end{equation}

    Let $\tilde{s}_{l,j,1}:= \inf\{kN: kN \geq t_j - l\}$, and 
    $\tilde{s}_{l,j,2}:= \sup\{kN: kN \leq t_{j+1} - l\}$. Then it is easy to verify that $\tilde{s}_{l,j,1} - (t_j - l) < N$ and $(t_{j+1} - l) - \tilde{s}_{l,j,2} < N$ holds for any $j \geq 0$. Moreover, since  $ \{\tilde{s}_{l,j,1}, \tilde{s}_{l,j,2} : j\geq 0\}\subset \{jN:j\geq 0\}$ for any $l\geq 0$,  together with the fact that sequence $\{i_k\}$ is drawn by reshuffling,  we can conclude that there exists a diminishing sequence $\{\hat{\delta}^{\star}_k\}$ such that 
    \begin{equation}
        \begin{aligned}
        & \frac{1}{\tilde{s}_{l,j,2} - \tilde{s}_{l,j,1}}\sum_{k = \tilde{s}_{l,j,1}}^{\tilde{s}_{l,j,2}-1}\D_{f_{i_{k-l}}}^{\delta_{k-l}^{\star}}(x_{t_j}) \subseteq   \mathrm{conv}\left(  \D_f^{\hat{\delta}^{\star}_{t_j}}(x_{t_j})\right).
        \end{aligned}
    \end{equation}

    Furthermore, from the fact that $\lim_{j\to +\infty}T_j =+\infty$ in Lemma \ref{Le_Defin_tk}, we can conclude that 
    \begin{equation}
        \begin{aligned}
            &\lim_{j\to +\infty} \norm{\frac{1}{T_j}\sum_{k = t_j}^{t_{j+1}-1} g_{k-l} - \frac{1}{\tilde{s}_{l,j,2} - \tilde{s}_{l,j,1}}\sum_{k = \tilde{s}_{l,j,1}}^{\tilde{s}_{l,j,2}-1} g_{k-l}}\\
            \leq{}& \lim_{j\to +\infty} \norm{\frac{1}{T_j}\sum_{k = t_j}^{t_{j+1}-1} g_{k-l} - \frac{1}{T_j}\sum_{k = \tilde{s}_{l,j,1}}^{\tilde{s}_{l,j,2}-1} g_{k-l}}  + \lim_{j\to +\infty} \norm{\left(\frac{1}{T_j} - \frac{1}{\tilde{s}_{l,j,2} - \tilde{s}_{l,j,1}} \right)\sum_{k = \tilde{s}_{l,j,1}}^{\tilde{s}_{l,j,2}-1} g_{k-l}}\\
            = {}& 0. 
        \end{aligned}
    \end{equation}
    Then we have that there exists a nonnegative diminishing sequence $\{\tilde{\delta}_k^*\}$ such that 
    \begin{equation}
        \frac{1}{T_j}\sum_{k = t_j}^{t_{j+1}-1} g_{k-l} \in \conv\left( \D_f^{\tilde{\delta}^*_{t_j}}(x_{t_{j}}) \right).
    \end{equation}

    Therefore, Lemma \ref{Le_approximate_evaluation} illustrates that there exists a non-negative diminishing sequence $\{\tilde{\delta}^{\star}_k\}$ such that 
    \begin{equation}
        \frac{1}{T_j}\sum_{k = t_j}^{t_{j+1}-1} g_{k-l} \in   \D_f^{\tilde{\delta}^{\star}_{t_j}}(x_{t_{j}}). 
    \end{equation}
    This completes the proof. 
\end{proof}

Based on Lemma \ref{Le_Defin_tk} and Lemma \ref{Le_aggregrate_gk}, we present Proposition \ref{Prop_aggregrate_mk} showing that the average of $\eta_k \mkp$ over $t_j \leq k \leq t_{j+1}$ is also an approximated evaluation to $\D_f(x_{t_j})$. 
\begin{prop}
    \label{Prop_aggregrate_mk}
    Suppose Assumption \ref{Assumption_f} and Assumption \ref{Assumption_Framework} hold, and $\lim_{k\to +\infty} \eta_k = 0$. Then there exists a non-negative diminishing sequence $\{\tilde{\delta}^{\star}_k\}$ such that
    \begin{equation}
         \frac{1}{T_j \eta_{t_j}}\sum_{k = t_j}^{t_{j+1}-1}  \eta_k \mkp \in   \D_f^{\tilde{\delta}^{\star}_{t_j}}(x_{t_{j}}) 
    \end{equation}
    holds for any $k\geq 0$. Here the sequence $\{t_j\}$ are chosen as \eqref{Eq_Le_Defin_tk_0}. 
\end{prop}
\begin{proof}
    From the update scheme of \eqref{Eq_Framework},  for any $k\geq 0$, it holds that 
    \begin{equation}
        \label{Eq_Prop_aggregrate_mk_2}
        \mkp =  \beta^{T_j+1}m_{k-T_j} + \sum_{l = k-T_j}^{k} (1-\beta)\beta^{k-l} g_{l}. 
    \end{equation}
    Substitute \eqref{Eq_Prop_aggregrate_mk_2} into the formulation of 
    $\sum_{k = t_j}^{t_{j+1}-1} m_{k+1}$, we have 
    \begin{equation}
        \begin{aligned}
            &\frac{1}{T_j}\sum_{k = t_j}^{t_{j+1}-1}  \mkp = \frac{1}{T_j}\sum_{k = t_j}^{t_{j+1}-1} \left(\beta^{T_j+1}m_{k-T_j} + \sum_{l = 0}^{T_j} (1-\beta)\beta^l g_{k-l}\right)\\
            ={}& \frac{1}{T_j}\sum_{k = t_j}^{t_{j+1}-1}\beta^{T_j+1}m_{k-T_j} +  \sum_{l = 0}^{T_j} (1-\beta)\beta^l \left(\frac{1}{T_j}\sum_{k = t_j}^{t_{j+1}-1}  g_{k-l}\right).
        \end{aligned}
    \end{equation}
    As Lemma \ref{Le_UB_gk} illustrates that sequence $\{\mk\}$ is uniformly bounded, it holds from the definition of sequences $\{T_k\}$ and $\{t_k\}$ that 
    \begin{equation}
        \lim_{j\to +\infty} \norm{\frac{1}{T_j}\sum_{k = t_j}^{t_{j+1}-1}\beta^{T_j+1}m_{k-T_j}} \leq \lim_{j\to +\infty} \beta^{T_j+1} \sup_{k\geq 0} \norm{\mk} = 0. 
    \end{equation}
    Moreover, from Assumption \ref{Assumption_Framework}(2) and Lemma \ref{Le_Defin_tk}, there exists a diminishing nonnegative sequence $\{\delta^{\star}_k\}$ such that for any $l\leq T_j$, $g_{k-l} \in \D_{f_{i_{k-l}}}^{\tilde{\delta}_{t_j}^{\star}}(x_{t_j})$. Then Lemma \ref{Le_aggregrate_gk} illustrates that 
    \begin{equation}
        \begin{aligned}
            &\sum_{l = 0}^{T_j} (1-\beta)\beta^l \left(\frac{1}{T_j}\sum_{k = t_j}^{t_{j+1}-1}  g_{k-l}\right) \in   \sum_{l = 0}^{T_j} (1-\beta)\beta^l \D_{f}^{\tilde{\delta}_{t_j}^{\star}}(x_{t_j}) 
            \subseteq (1 - \beta^{T_j + 1})\conv\left(  \D_{f}^{\tilde{\delta}_{t_j}^{\star}}(x_{t_j}) \right). 
        \end{aligned}
    \end{equation}
    Then based on Lemma \ref{Le_approximate_evaluation} and Lemma \ref{Le_aggregrate_gk}, we can conclude that there exists a non-negative diminishing sequence $\{\tilde{\delta}_k\}$ such that 
    \begin{equation}
        \label{Eq_Prop_aggregrate_mk_0}
        \frac{1}{T_j}\sum_{k = t_j}^{t_{j+1}-1}  \mkp \in \D_{f}^{\tilde{\delta}_{t_j}}(x_{t_j}).
    \end{equation}

    On the other hand, notice that 
    \begin{equation}
        \begin{aligned}
            & \norm{\frac{1}{T_j \eta_{t_j}}\left( \sum_{k = t_j}^{t_{j+1}-1}  \eta_k \mkp \right) - \frac{1}{T_j} \left( \sum_{k = t_j}^{t_{j+1}-1}   \mkp \right)}=\frac{1}{T_j }\norm{\left( \sum_{k = t_j}^{t_{j+1}-1}  \frac{\eta_k}{\eta_{t_j}} \mkp \right) - \left( \sum_{k = t_j}^{t_{j+1}-1}   \mkp \right)}\\
            \leq{}& \frac{1}{T_j} \sum_{k = t_j}^{t_{j+1}-1} \norm{\left(\frac{\eta_k}{\eta_{t_j}} - 1\right) \mkp}  \leq  \frac{1}{T_j } \sup_{k\geq 0} \norm{\mk}  \sum_{k = t_j}^{t_{j+1}-1}  \left|\frac{\eta_k}{\eta_{t_j}} - 1\right|\\
            \leq{}& \sup_{k\geq 0} \norm{\mk} \sup_{t_k \leq k< t_{j+1}} \left|\frac{\eta_k}{\eta_{t_j}} - 1\right|. 
        \end{aligned}
    \end{equation}
    Then it directly follows from Lemma \ref{Le_Defin_tk} and the uniform boundedness of sequence $\{\mk\}$ in Lemma \ref{Le_UB_gk} that 
    \begin{equation}
        \label{Eq_Prop_aggregrate_mk_1}
        \begin{aligned}
            &\lim_{j\to +\infty} \norm{\frac{1}{T_j \eta_{t_j}}\left( \sum_{k = t_j}^{t_{j+1}-1}  \eta_k \mkp \right) - \frac{1}{T_j} \left( \sum_{k = t_j}^{t_{j+1}-1}   \mkp \right)} \\
            \leq{}& \lim_{j\to +\infty} \sup_{k\geq 0} \norm{\mk} \sup_{t_k \leq k< t_{j+1}} \left|\frac{\eta_k}{\eta_{t_j}} - 1\right| =  0. 
        \end{aligned}
    \end{equation}
    As a result, by combining \eqref{Eq_Prop_aggregrate_mk_0} and \eqref{Eq_Prop_aggregrate_mk_1} together, we can conclude that there exists a non-negative diminishing sequence $\{\tilde{\delta}^{\star}_k\}$ such that
    \begin{equation}
         \frac{1}{T_j \eta_{t_j}}\sum_{k = t_j}^{t_{j+1}-1}  \eta_k \mkp \in   \D_f^{\tilde{\delta}^{\star}_{t_j}}(x_{t_{j}}) 
    \end{equation}
    holds for any $k\geq 0$.  This completes the proof. 
\end{proof}

Then by utilizing the concept of interpolated process in Definition \ref{Defin_interpolated_process}, the following proposition illustrates that the sequence $\{x_{t_j}\}$ asymptotically approximates the following differential inclusion 
\begin{equation}
    \label{Eq_DI_SGD}
    \frac{\mathrm{d}x}{\mathrm{d}t} \in -\D_f(x).
\end{equation}
\begin{prop}
    \label{Prop_case2_DI}
    Suppose Assumption \ref{Assumption_f} and Assumption \ref{Assumption_Framework} hold, and $\lim_{k\to +\infty} \eta_k = 0$. Then the interpolated process of $\{x_{t_j}\}$ with respect to the learning rates $\{T_j \eta_{t_j}\}$ is a perturbed solution of the differential inclusion \eqref{Eq_DI_SGD}. Here the sequence $\{t_j\}$ is chosen as in \eqref{Eq_Le_Defin_tk_0}. 
\end{prop}
\begin{proof}
    From the update scheme in \eqref{Eq_Framework}, for any $j > 0$, it holds that, $x_{t_{j+1}} - x_{t_j} = -\sum_{k = t_j}^{t_{j+1}-1}  \eta_k \mkp$.  
    Then together with Proposition \ref{Prop_aggregrate_mk}, we can conclude that there exists a non-negative diminishing sequence $\{\tilde{\delta}^{\star}_k\}$ such that $x_{t_{j+1}} - x_{t_j} \in T_j \eta_{t_j}  \D_f^{\tilde{\delta}^{\star}_{t_j}}(x_{t_{j}})$.  

    Notice that the definition of $\{t_j\}$ illustrates that $\lim_{j\to +\infty} T_j \eta_{t_j} = 0$. 
    Therefore, from the definition of the perturbed solution, we can conclude that the interpolated process of $\{x_{t_j}\}$ with respect to the learning rates $\{T_j \eta_{t_j}\}$ is a perturbed solution of the differential inclusion  \eqref{Eq_DI_SGD}. This completes the proof. 
\end{proof}

The following proposition illustrates that $f$ is a Lyapunov function for the differential inclusion \eqref{Eq_DI_SGD}, which directly follows from the results in \cite[Section 6.1]{bolte2021conservative}. 
\begin{prop}
    \label{Prop_Lyapunov_SGD}
    Suppose Assumption \ref{Assumption_f} holds, then $f$ is a Lyapunov function for the differential inclusion \eqref{Eq_DI_SGD} with the stable set $\{x \in \Rn: 0\in \D_f(x) \}$. 
\end{prop}

\paragraph{Proof for Theorem \ref{The_convergence}}
When $\mathop{\lim\inf}_{k\to +\infty} \eta_k > 0$, Proposition \ref{Prop_case1_etak_bounded} directly shows that any cluster point of sequence $\{\xk\}$ is contained in the subset $\{x \in \Rn: 0 \in \D_f(x)\}$, and the sequence $\{f(\xk)\}$ converges. 

On the other hand, when $\lim_{k\to +\infty} \eta_k = 0$, Proposition \ref{Prop_case2_DI} illustrates that any interpolated process of $\{\xk\}$ is a perturbed solution of the differential inclusion \eqref{Eq_DI_SGD}. Then based on Proposition \ref{Prop_Lyapunov_SGD}, \cite[Proposition 3.33]{benaim2005stochastic} and \cite[Theorem 4.2]{benaim2005stochastic}, we can conclude that any cluster point of sequence $\{\xk\}$ is contained in the subset $\{x \in \Rn: 0 \in \D_f(x)\}$, and the sequence $\{f(\xk)\}$ converges. This completes the proof.

\subsection{Remarks on stability under coercivity condition}

In this subsection, we focus on the following deterministic version of the framework \eqref{Eq_Framework}, 
\begin{equation}
\tag{D-LFGD}
\label{Eq_Framework_determin}
    \left\{
    \begin{aligned}
        &g_k \in \D_f(\xk), \\
        &\eta_k = \frac{\mu_k}{\sqrt{\varepsilon_0 + \sum_{i = 1}^{k} \tau_i\norm{m_i}^2}},\\
        &\mkp = \beta \mk + (1-\beta) g_k,\\
        &\xkp = \xk - \eta_k \mkp, 
    \end{aligned}
    \right.
\end{equation}
and we aim to investigate a sufficient condition to guarantee the uniform boundedness of the sequence $\{\xk\}$  in \eqref{Eq_Framework_determin}. Although some existing works \cite{bianchi2022convergence,josz2023global,josz2023lyapunov} have established the uniform boundedness of the sequence generated by some particular stochastic subgradient methods, their analyses are restricted to those methods where the stepsizes are fixed as a constant throughout the iterations. As a result, their analyses cannot be directly applied to analyze the framework \eqref{Eq_Framework_determin}. 

Motivated by the techniques developed in \cite{josz2023global,josz2023lyapunov}, we aim to prove that the sequence $\{\xk\}$ generated by \eqref{Eq_Framework_determin} is uniformly bounded under mild conditions. 
Given a locally bounded set-valued mapping $\D: \Rn \rightrightarrows \Rn$ and the sequence of stepsizes $\{\eta_k\}$,  we first consider the following discrete-time iteration:
\begin{equation}
    \label{Eq_stability_iterates}
    \xkp \in \xk - \eta_k \conv\left(\D^{\delta_k}(\xk) \right), 
\end{equation}
and its corresponding differential inclusion
\begin{equation}
    \label{Eq_stability_DI}
    \frac{\mathrm{d}x}{\mathrm{d}t} \in \D(x). 
\end{equation}
Let $\hat{x}(t)$ be the interpolated process of the sequence $\{\xk\}$ generated by \eqref{Eq_stability_iterates} with respect to the sequence of stepsizes $\{ \eta_k\}$. Moreover, let  $\lambda: \bb{N}_+ \to \bb{R}_+, \Lambda: \bb{R}_+ \to \bb{N}_+$ be defined by $\lambda(0) := 0$, $\lambda(i) := \sum_{k = 0}^{i-1} \eta_k$, and $\Lambda(t) := \sup  \{k \geq 0: t\geq \lambda(k)\} $.

In the following, we present Lemma \ref{Le_stability_close_DI_iter} as an extension to \cite[Lemma 1]{josz2023lyapunov},  showing the relationship between the discrete iteration sequence \eqref{Eq_stability_iterates} and the trajectories of the continuous-time differential inclusion \eqref{Eq_stability_DI}. The proof of Lemma \ref{Le_stability_close_DI_iter} directly follows the same proof techniques as in \cite[Lemma 1]{josz2023lyapunov}, hence is omitted for simplicity. 
\begin{lem}
    \label{Le_stability_close_DI_iter}
    Let $\D: \Rn \rightrightarrows \Rn$ be a graph-closed locally bounded convex-valued set-valued mapping. Let $X_0$ be a compact subset of $\Rn$ and let $T>0$ be a prefixed constant. Suppose that there exist constants $\bar{\alpha}, r > 0$ such that for any $\alpha \in (0, \bar{\alpha}]$ and for any sequence $\{\xk\}$ generated by \eqref{Eq_stability_iterates} with $x_0 \in X_0$ and $\sup_{k\geq 0}(\eta_k + \delta_k) \leq \alpha$, we have that 
    $\max\{\norm{\xk}: k\leq \Lambda(T)\} \leq r$. 
    
    Then for any $\varepsilon > 0$, there exists $\hat{\alpha}\in (0, \bar{\alpha}]$ such that for all $\alpha \leq \hat{\alpha}$ and for any sequence $\{\xk\}$ generated by \eqref{Eq_stability_iterates} with $x_0 \in X_0$ and $\sup_{k\geq 0}\eta_k \leq \alpha$, there exists a trajectory $\bar{x}$ of the differential inclusion \eqref{Eq_stability_DI} with $\bar{x}(0) \in X_0$, such that 
    \begin{equation}
        \sup_{0\leq t\leq T}\norm{\bar{x}(t) - \tilde{x}(t)} \leq \varepsilon.
    \end{equation}
    Here $\tilde{x}$ is an interpolated process of sequence $\{\xk\}$ with respect to the sequence of stepsizes $\{\eta_k\}$. 
\end{lem}

Now we make the following assumptions on the definability and coercivity of $f$. 
\begin{assumpt}
    \label{Assumption_f_definable}
    \begin{enumerate}
        \item $f$ is a locally Lipschitz continuous definable function that admits $\tilde{\D}_f$ as its definable conservative field.  
        \item $f$ is coercive over $\Rn$. That is, for any sequence $\{\yk\} $ such that $\lim_{k\to +\infty}\norm{\yk} = +\infty$, it holds that $\lim_{k\to +\infty} f(\yk) = +\infty$. 
        \item The sequences $\{\tau_k\}$ and $\{\mu_k\}$ in \eqref{Eq_Framework} are positive, non-decreasing. Moreover, there exists a locally bounded function $\psi: \Rn \to \bb{R}_+$ such that $\max\{\tau_k , \mu_k \} \leq \sup_{0\leq i\leq k}\psi(x_i)$ holds for any $k\geq 0$. 
    \end{enumerate}
\end{assumpt}

Here are some remarks on Assumption \ref{Assumption_f_definable}. As demonstrated in \cite{bolte2021conservative,castera2021inertial,josz2023global,josz2023lyapunov}, the assumptions on definibility of $f$ and $\D_f$ is mild in practice. Moreover, Assumption \ref{Assumption_f_definable}(2) assumes the coercivity of the objective function, which is also mild in various deep learning tasks, especially when quadratic regularization (i.e., weight decay) is employed. Furthermore, Assumption \ref{Assumption_f_definable}(3) characterizes how we construct the sequences $\{\tau_k\}$ and $\{\mu_k\}$, which align with the update schemes of DoG and DoWG. 

Based on Assumption \ref{Assumption_f_definable}, in this subsection, we can choose $\D_f$ as 
\begin{equation}
    \D_f(x) = \conv(\tilde{\D}_f(x)). 
\end{equation}
Then it straghtforwardly follows from Proposition \ref{Prop_definable_regularity} that the set $\{f(x) : 0 \in \D_f(x)\}$ is finite.

The following proposition illustrates that the sequence $\{ \xk\}$  approximates the subgradient descent method with a controlled approximation error. 
\begin{prop}
    \label{Prop_stable_mk_close}
    Suppose Assumption \ref{Assumption_f_definable} holds, and $\sup_{0\leq i\leq (k+1)q} \norm{\D_f(x_i)} + \norm{m_0} \leq M_f$. Then for any $k\geq 0$ and any $q \in \bb{N}_+$, it holds that 
    \begin{equation}
        \mkp \in \conv\left(\D_f^{\delta_{q,k}}(\xk)\right). 
    \end{equation}
    Here $\delta_{q,k} =  \beta^{q+1}M_f +  \frac{M_f}{\sqrt{\varepsilon_0}} \sum_{i = k-q}^{k-1}  \mu_i$. 
\end{prop}
\begin{proof}
    Notice that $\sup_{0\leq i\leq (k+1)q} \norm{g_i} + \norm{m_0} \leq M_f$ implies that $\sup_{0\leq i\leq k}\norm{\mkp} \leq M_f$. Therefore, for any $k, q \in \bb{N}_+$ and any $i \in [k-q,k]$, it directly follows from the update scheme in \eqref{Eq_Framework_determin} that 
    \begin{equation}
        \begin{aligned}
            \norm{x_i - \xk} \leq \sum_{j = i}^{k-1} \norm{\eta_j m_{j+1}} \leq \sum_{j = i}^{k-1} \eta_j M_f \leq  \sum_{j = i}^{k-1} \frac{ \mu_j}{\sqrt{\varepsilon_0}} M_f \leq \sum_{j = k-q}^{k-1} \frac{ \mu_j}{\sqrt{\varepsilon_0}} M_f.
        \end{aligned}
    \end{equation}

    On the other hand, it follows from the update scheme of the sequence $\{\mk\}$ in \eqref{Eq_Framework_determin} that
    \begin{equation}
        \begin{aligned}
            \mkp ={}& \beta^{q+1} m_{k-q} + (1-\beta) \sum_{j = k-q}^k \beta^{(k-j)} g_j\\
            \in{}&  \beta^{q+1} m_{k-q} + (1-\beta) \sum_{j = k-q}^k \beta^{(k-j)} \D_f^{\delta_{k}}(\xk),
        \end{aligned}
    \end{equation}
    where $\delta_k = \sum_{j = k-q}^{k-1} \frac{ \mu_j}{\sqrt{\varepsilon_0}} M_f$. Notice that $\sup_{0\leq i\leq (k+1)q} \norm{\D_f(x_i)}  \leq M_f$, it holds that  $\sup_{0\leq i\leq k} \norm{m_i} \leq M_f$. As a result,  we can conclude that 
    \begin{equation}
        \mkp \in \conv \left(\D_f^{\delta_{q,k}}(\xk) \right),
    \end{equation}
    with $\delta_{q,k} =  \beta^{q+1}M_f +  \frac{M_f}{\sqrt{\varepsilon_0}} \sum_{i = k-q}^{k-1}  \mu_i$.  
\end{proof}

Denote $\ca{L}_{r}:= \{x \in \Rn: f(x) \leq r\}$ as the level set for $f$ in \eqref{Prob_Ori}. Then, from the coercivity of $f$ assumed in Assumption \ref{Assumption_f_definable}, we can conclude that $\ca{L}_{r}$ is a closed compact subset of $\Rn$ for any $r > 0$. 

The following theorem illustrates that the global stability of the framework \eqref{Assumption_f_definable}, in the sense that the sequence $\{\xk\}$ generated by \eqref{Eq_Framework_determin} is uniformly bounded under Assumption \ref{Assumption_f_definable} with a sufficiently small scaling parameter $\{\mu_k\}$. It is worth mentioning that although the global stability of momentum SGD with fixed stepsizes is analyzed in \cite{josz2023global}, their implementations of momentum SGD are different from the momentum SGD implemented in PyTorch and TensorFlow. Hence, their techniques cannot be applied to explain the global stability of \eqref{Eq_Framework_determin}.

\begin{theo}
    \label{The_stable_corecive}
    Suppose Assumption \ref{Assumption_f_definable} holds.  Then for any compact subset $X_0 \subset \Rn$, there exist constants $\alpha > 0$ and $\tilde{r} > 0$ such that whenever $\sup_{k\geq 0} \mu_k  \leq \alpha$, the sequence $\{\xk\}$ generated by \eqref{Eq_Framework_determin} with $x_0 \in X_0$ is restricted in $\ca{L}_{\tilde{r}}$. 
\end{theo}
\begin{proof}

    From the definability of $f$ and $\D_f$, we can conclude that the set $\{f(x): 0 \in \D_f(x) \}$ is finite. Let $r_0$ be a positive constant such that  
    \begin{equation}
        r_0 >  \sup_{x \in X_0} f(x), \quad r_0> \max \{f(x):  0 \in \D_f(x) \}, 
    \end{equation}
    and  denote $M_f:= \sup_{x \in \ca{L}_{4r_0}} \norm{\D_f(x)}$. From the local boundedness of $\D_f$, for any compact subset $\hat{X}_0 \subseteq \ca{L}_{3 r_0}$, any positive sequences $\{\eta_k\}$ and $\{\delta_k\}$, and any sequence $\{z_k\}$ satisfying the update scheme,
    \begin{equation}
        \zkp \in \zk - \eta_k \D_f^{\delta_k}(\zk),\quad z_0 \in \hat{X}_0,
    \end{equation} 
    there exist $T \in (0, \frac{r_0}{4(M_f +1)^2})$ and $\alpha_z > 0$ such that whenever $\sup_{k\geq 0} \eta_k + \delta_k \leq \alpha_z$,  we have $\{z_k: \lambda(k) \leq T\} \subset \ca{L}_{4r_0}$. Then, with that particular choice of $T$, we  choose $\iota:= \inf_{x \in \ca{L}_{4r_0} \setminus \ca{L}_{r_0}} \inf\{\norm{d}^2: d \in \D_f (x)\}$, and set 
    \begin{equation}
        \varepsilon :=  \min\left\{\frac{\min\{1, \iota T \}}{2M_f}, \frac{\mathrm{dist}(\ca{L}_{r_0}, \ca{L}_{2r_0}\setminus \ca{L}_{r_0})}{4M_f}  \right\}. 
    \end{equation} 
     Then from Lemma \ref{Le_stability_close_DI_iter}, there exists $\alpha_z > 0$ such that for any sequence $\{z_k\}$ satisfying the update scheme, $\zkp \in \zk - \sigma_k \D_f^{\delta_k}(\zk)$ 
    with $z_0 \in \ca{L}_{2r_0}$ and $\sup_{k\geq 0}(\sigma_k + \delta_k) \leq \alpha_z$, there exists a trajectory $\bar{z}$ of the differential inclusion \eqref{Eq_DI_SGD} 
    with $\bar{z}(0) \in \ca{L}_{2r_0}$, such that $\sup_{0\leq t\leq T}\norm{\bar{z}(t) - \hat{z}(t)} \leq \frac{1}{2}\varepsilon$. 
    Here $\hat{z}$ is an interpolated process of the sequence $\{\zk\}$ with respect to the stepsizes $\{\sigma_k\}$. 
 
    Now choose $q> \frac{4|\log(\alpha_z M_f^{-1})|}{|\log(\beta)|}$ and $\alpha < \frac{\alpha_z\sqrt{\varepsilon_0}}{4qM_f}$. From Proposition \ref{Prop_stable_mk_close}, there exists a nonnegative sequence $\{\delta_{q,k}\}$ such that the update scheme of $\{ \xk\}$ can be described by $\xkp \in  \xk - \eta_k \D_f^{\delta_{q,k}}( \xk)$ 
    with $\delta_{q,k}+ \eta_k \leq \alpha_z$ holds for any $k\geq 0$. 

    From Lemma \ref{Le_stability_close_DI_iter}, for any $i \geq 0$ and any  $j\leq \Lambda(\lambda(i)+ T)$, when $ x_i \in \ca{L}_{2r_0}$, there exists a trajectory $\bar{z}$ of the differential inclusion  \eqref{Eq_DI_SGD}  such that $\bar{z}(0) \in \ca{L}_{2r_0}$ and $\sup_{ 0\leq t\leq T} \norm{\bar{z}(t) - \hat{x}(t)} \leq \varepsilon$. 
    Here $\hat{x}$ is an interpolated process of $\{ \xk: k\geq i\}$ with respect to $\{\eta_k\}$. 
    Then we have
    \begin{equation}
        \begin{aligned}
            &f( x_j) \leq f(\bar{z}(\lambda(j) - \lambda(i) ))  + M_f \varepsilon = f(\bar{z}(0)) +M_f \varepsilon +  \int_{0}^{\lambda(j) - \lambda(i)} -\mathrm{dist}(0, \D_f(\bar{z}(s)))^2 \mathrm{d}s \\
            \leq{}& f(\bar{z}(0)) +M_f \varepsilon \leq f( x_i) + 2M_f \varepsilon \leq 2r_0 + 2 M_f \varepsilon \leq 3r_0. 
        \end{aligned}
    \end{equation}
    This illustrates that $\{ \xk: i\leq k\leq \Lambda(\lambda(i)+ T) \} \subset \ca{L}_{3r_0}$.

    On the other hand,  when $ x_i \in \ca{L}_{3r_0} \setminus \ca{L}_{2r_0}$, from the choice of $T$ we can conclude that $\{ x_j: i \leq j \leq \Lambda(\lambda(i)+ T)\} \cap  \ca{L}_{r_0} = \emptyset$. Then by using Lemma \ref{Le_stability_close_DI_iter} again, there exists a trajectory $\bar{z}$ such that $\bar{z}(0) \in \ca{L}_{3r_0}$ and  $\sup_{0\leq t\leq T} \norm{\bar{z}(t) - \hat{x}(t)} \leq \varepsilon$. Here $\hat{x}$ is an interpolated process of $\{ \xk: k\geq i\}$.

    Since $ x_i \in \ca{L}_{3r_0} \setminus \ca{L}_{2r_0} $, then we can conclude that $\mathrm{dist}(\bar{z}(0), \ca{L}_{r_0}) \geq T M_f$, hence $\{\bar{z}(t): 0\leq t\leq T\} \cap \ca{L}_{r_0} = \emptyset$. As a result, we have
    \begin{equation}
        \begin{aligned}
        &f( x_{\Lambda(\lambda(i)+ T)}) \leq f(\bar{z}(T)) + \varepsilon M_f = f(\bar{z}(0)) +\varepsilon M_f + \int_{0}^{T} -\mathrm{dist}(0, \D_f(\bar{z}(s)))^2 \mathrm{d}s\\
        \leq{}& f(\bar{z}(0)) - T \iota + \varepsilon M_f \leq f(\bar{z}(0)) - \varepsilon M_f \leq 3r_0 - \varepsilon M_f. 
        \end{aligned}
    \end{equation}

    Therefore, we can conclude that there exists a strictly increasing sequence $\{k_l\} \subseteq \bb{N}_+$ such that $\{ x_{k_l}\} \in \ca{L}_{3r_0}$, and $k_{l+1} - k_l \leq \Lambda(\lambda(k_l) + T)$. Notice that 
    \begin{equation}
        \sup_{k\geq 0} \norm{ \xk} \leq \sup_{l\geq 0} \norm{ x_{k_l}} + T M_f < +\infty,
    \end{equation}
    we can conclude that when $ x_{0} \in X_0$, the sequence $ \sup_{k\geq 0}  f(\xk ) \leq 3r_0 + T M_f^2 < 4r_0  $, hence sequence $\{\xk\}$ is uniformly bounded and restricted in $\ca{L}_{4r_0}$. As a result, let $\tilde{r} = 4r_0$, we complete the proof. 
\end{proof}

\section{Applications}

This section demonstrates the practical implementation of our theoretical analysis on the framework \eqref{Eq_Framework}. In Section 4,1, we show that the DoG method \cite{ivgi2023dog} fits into the framework \eqref{Eq_Framework}, hence enjoys the convergence guarantees in training nonsmooth neural networks with random reshuffling. Moreover, we propose a novel learning-rate-free momentum SGD, and show its convergence properties based on the results in Section 4.2.

\subsection{Convergence guarantees for DoG}
In this subsection, we show that our proposed framework \eqref{Eq_Framework} encompasses the DoG with random reshuffling for solving \eqref{Prob_Ori}. Following the update schemes in \cite{ivgi2023dog}, we present the detailed algorithm in Algorithm \ref{Alg:DOG}. 

\begin{algorithm}[htbp]
	\begin{algorithmic}[1]  
		\Require Initial point $x_0 \in \Rn$, parameters  $\varepsilon_0 > 0$, $\rho > 0$.
		\For{$i = 0,1,2,...$}
                \State Sample $\{\pi_{i,0},..., \pi_{i,N-1}\}$ as random permutation of $[N]$. 
                \For{$j = 0,..., N-1$}
        		\State Compute $g_{i,j} \in \D_{f_{\pi_{i,j}}}(x_{i,j})$;
        		\State Compute $\mu_{i,j} = \rho \cdot \max_{i^* \leq i,~ j^*\leq j}\{\norm{x_{i^*,j^*} - x_0}\}$;
                    \State Compute the learning rate $\eta_{i,j} =  \frac{\mu_{i,j}}{ \sqrt{\varepsilon_0 +\sum_{i^* \leq i,~ j^*\leq j}\norm{g_{i^*,j^*}}^2 }}$;
        		\State Update the primal variables $x$ by $x_{i, j+1} = x_{i, j} - \eta_{i,j} g_{i,j}$;
                \EndFor
                \State $x_{i+1,0} = x_{i,N}$. 
		\EndFor
	\end{algorithmic}  
	\caption{The Distance over Gradients method (DoG) with random reshuffling.}  
	\label{Alg:DOG}
\end{algorithm}

Then based on Theorem \ref{The_convergence}, we present the convergence properties of Algorithm \ref{Alg:DOG} in the following theorem. The proof of Theorem \ref{Theo_convergence_DoG} directly follows from Theorem \ref{The_convergence} by choosing $\mu_k = \sup_{i\leq k} \norm{x_i - x_0}$ and $\tau_k = 1$ in \eqref{Eq_Framework}, hence is omitted for simplicity. 
\begin{theo}
    \label{Theo_convergence_DoG}
    Suppose Assumption \ref{Assumption_f} holds, and there exists a constant $M> 0$ such that $\sup_{k\geq 0} \norm{\xk} \leq M$ holds almost surely. Then almost surely, any cluster point of $\{\xk\}$ generated by Algorithm \ref{Alg:DOG} lies in the subset $\{x \in \Rn: 0 \in \D_f(x)\}$, and the sequence $\{f(\xk)\}$ converges.  
\end{theo}

The numerical experiments in \cite{ivgi2023dog} show that in the training of deep neural networks, DoG has comparable performance as SGD with a well-tuned learning rate. The theoretical framework in \cite{ivgi2023dog} requires the convexity of the objective function $f$, a condition that cannot be satisfied in neural network training tasks. Our findings in Theorem \ref{Theo_convergence_DoG} explain the convergence behavior of DoG when applied to the training of nonsmooth neural networks, thereby helping to understand the empirical success of these learning-rate-free optimization methods in neural network training tasks.

\begin{rmk}
    As demonstrated in \cite{khaled2023dowg}, we can recover the DoWG by replacing Step 7 of Algorithm \ref{Alg:DOG} with the following scheme, 
    \begin{equation}
        \eta_{i,j} =  \frac{\mu_{i,j}^2}{ \sqrt{\varepsilon_0 +\sum_{i^* \leq i,~ j^*\leq j}\mu_{i,j}^2\norm{g_{i^*,j^*}}^2 }}.
    \end{equation}
    Then it is easy to verify that such an algorithm aligns with the framework \eqref{Eq_Framework}, by choosing $\mu_k = \tau_k = (\max_{i\leq k}\norm{x_i - x_0})^2$ and $\beta = 0$. Therefore, similar to the results in Theorem \ref{Theo_convergence_DoG}, our proposed framework \eqref{Eq_Framework} is capable of providing theoretical convergence guarantees for DoWG in training nonsmooth neural networks, which extends the results from \cite{khaled2023dowg}.  
\end{rmk}

\subsection{Learning-rate-free Momentum SGD}

In this subsection, we develop a learning-rate-free momentum SGD with convergence guarantees in training nonsmooth neural networks. Motivated by the DoG method proposed by \cite{ivgi2023dog}, we consider choosing $\mu_k = \max_{i\leq k}\norm{x_i - x_0}$, $\tau_k = 1$ and $\beta \in (0,1)$ in \eqref{Eq_Framework}.  The detailed algorithm, named learning-rate-free momentum SGD (LFM),  is presented in Algorithm \ref{Alg:PFSGD}. 

In our proposed LFM, we employ the momentum terms $\{\mk\}$ for acceleration, which makes LFM different from the DoG and DoWG methods. Although the heavy-ball momentum is also employed in D-adapted SGD proposed in \cite[Algorithm 4]{defazio2023learning}, that method does not have any convergence guarantee, even in minimizing convex objective functions.

\begin{algorithm}[htbp]
	\begin{algorithmic}[1]  
		\Require Initial point $x_0 \in \Rn$, and $m_0 \in \Rn$, parameters  $\varepsilon_0 > 0$ and $\beta \in [0,1)$;
		\For{$i = 0,1,2,...$}
                \State Sample $\{\pi_{i,0},..., \pi_{i,N-1}\}$ as random permutation of $[N]$; 
                \For{$j = 0,..., N-1$}
        		\State Compute $g_{i,j} \in \D_{f_{\pi_{i,j}}}(x_{i,j})$;
        		\State Compute $\mu_{i,j} = \rho \cdot \max_{i^* \leq i,~ j^*\leq j}\{\norm{x_{i^*,j^*} - x_0}\}$;
                    \State Compute the learning rate $\eta_{i,j} = \frac{\mu_{i,j}}{ \sqrt{\varepsilon_0 +\sum_{i^* \leq i,~ j^*\leq j}\norm{m_{i^*,j^*}}^2 }}$;
        		\State Update the momentum term by $m_{i,j+1} = \beta m_{i,j} + (1-\beta) g_{i,j}$;
        		\State Update the primal variables $x$ by $x_{i, j+1} = x_{i, j} - \eta_{i,j} m_{i,j+1}$;
                \EndFor
                \State $x_{i+1,0} = x_{i,N}$, $m_{i+1,0} = m_{i,N}$. 
		\EndFor
	\end{algorithmic}  
	\caption{A learning-rate-free momentum-accelerated SGD (LFM) for solving \eqref{Prob_Ori}.}  
	\label{Alg:PFSGD}
\end{algorithm}

For the convergence properties of our proposed LFM, we present the following theorem illustrating that our proposed LFM can find a $\D_f$-stationary point of $f$ almost surely. The proof of Theorem \ref{Theo_convergence_LFM} directly follows from the results in Theorem \ref{The_convergence} by choosing $\mu_k = \sup_{i\leq k}\norm{x_i - x_0}$ and $\tau_k = 1$ in the framework \eqref{Eq_Framework}. 

\begin{theo}
    \label{Theo_convergence_LFM}
    Suppose Assumption \ref{Assumption_f} holds, and there exists a constant $M> 0$ such that  $\sup_{k\geq 0} \norm{\xk} \leq M$ almost surely. Then almost surely, any cluster point of sequence $\{\xk\}$ generated by Algorithm \ref{Alg:PFSGD} lies in the subset $\{x \in \Rn: 0 \in \D_f(x) \}$, and the sequence $\{f(\xk)\}$ converges. 
\end{theo}

\section{Numerical Experiments}

In this section, we evaluate the numerical performance of our proposed learning-rate-free momentum SGD (LFM), as outlined in Algorithm \ref{Alg:PFSGD}, for training nonsmooth neural networks.   All numerical experiments in this section are conducted on a server equipped with an Intel Xeon 6342 CPU and $6$ NVIDIA GeForce RTX 3090 GPUs, running Python 3.8 and PyTorch 1.9.0. 

In our numerical experiments, we focus on image classification tasks that train nonsmooth neural networks on the CIFAR datasets \cite{krizhevsky2009learning} and Imagenet \cite{deng2009imagenet}. It is important to note that we utilize the Rectified Linear Unit (ReLU) as the activation function for all networks, including ResNet50 \cite{he2016deep}, ResNet18, VGG-Net, and MobileNet. Consequently, the corresponding objective function is definable \cite{bolte2020mathematical} but exhibits non-smoothness and lacks Clarke regularity.

\begin{figure}[!htb]
	\centering
	\subfigure[LFM on CIFAR-10 with $\beta = 0.9$]{
		\begin{minipage}[t]{0.40\linewidth}
			\centering
			\includegraphics[width=\linewidth]{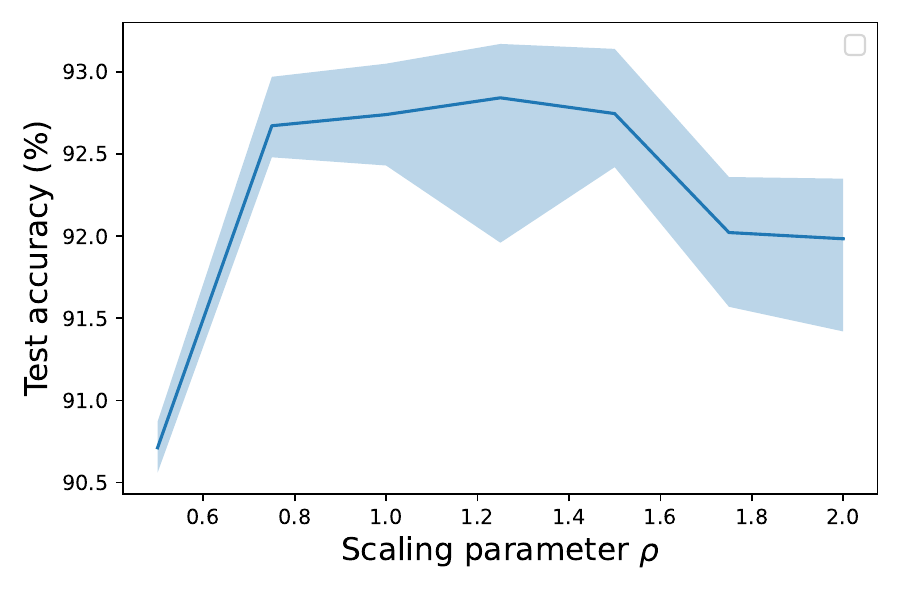}
			\label{Fig:Fig_lr_0_testacc}
		\end{minipage}%
	}%
	\subfigure[LFM on CIFAR-100  with $\beta = 0.9$]{
		\begin{minipage}[t]{0.40\linewidth}
			\centering
			\includegraphics[width=\linewidth]{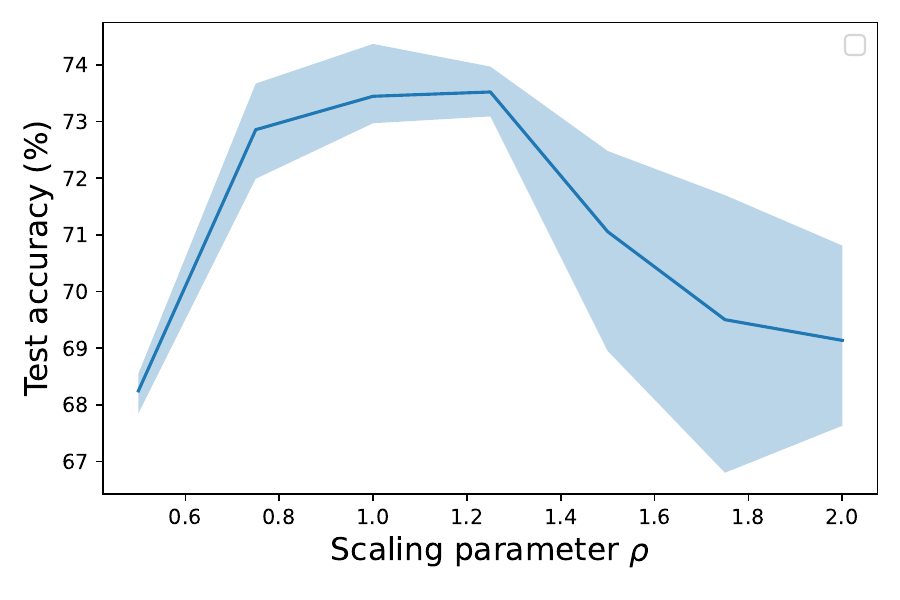}
			\label{Fig:Fig_lr_1_testacc}
		\end{minipage}%
	}%

        \subfigure[LFM on CIFAR-10 with $\rho = 1$]{
		\begin{minipage}[t]{0.40\linewidth}
			\centering
			\includegraphics[width=\linewidth]{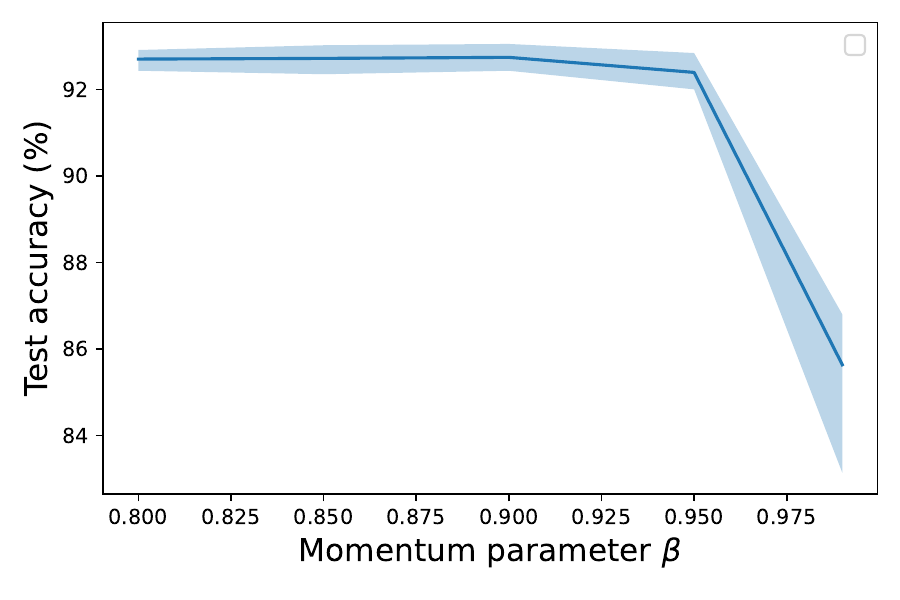}
			\label{Fig:Fig_moment_0_testacc}
		\end{minipage}%
	}%
	\subfigure[LFM on CIFAR-100 with $\rho = 1$]{
		\begin{minipage}[t]{0.40\linewidth}
			\centering
			\includegraphics[width=\linewidth]{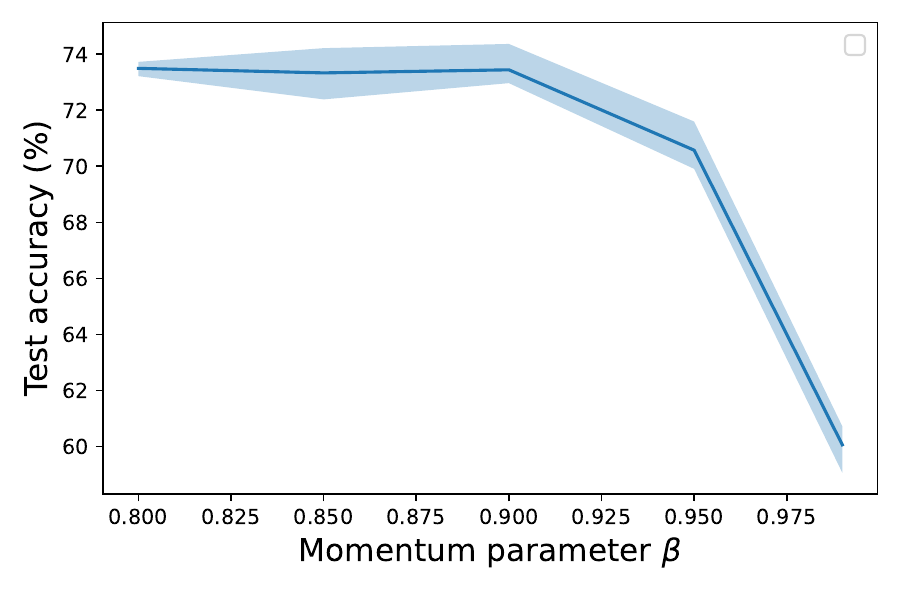}
			\label{Fig:Fig_moment_1_testacc}
		\end{minipage}%
	}%
	\caption{Numerical results on the final-epoch test accuracy of LFM with different scaling parameters $\rho$ and momentum parameters $\beta$.}
        \label{Fig_Param_LFM}
\end{figure}

\subsection{Choice of parameters in LFM}
We first investigate the performance of LFM under different choices of scaling parameter $\rho$ and momentum parameter $\beta$. We conducted evaluations on image classification tasks utilizing the CIFAR datasets and the ResNet50 model. For these evaluations, we selected values for $\rho$ from the set $\{0.5, 0.75, 1.0, 1.25, 1.5, 1.75, 2.0\}$ and for $\beta$ from the set $\{0.8, 0.85, 0.9, 0.95, 0.99\}$. In each test instance, we run LFM $5$ times with different random seeds. Figure \ref{Fig_Param_LFM} exhibits the final-epoch test accuracy of the training results yielded by LFM with different choices of $\rho$ and $\beta$. From the results in Figure \ref{Fig_Param_LFM}, LFM performs well when $\rho$ is chosen between $0.75$ and $1.25$, which is a wide enough interval for choosing the scaling parameter $\rho$. This robustness in parameter tuning distinctly sets LFM apart from algorithms such as Adam and SGD, which typically require an extensive grid search to determine a proper learning rate. Furthermore, from Figure \ref{Fig_Param_LFM}, we observe that the performance of LFM is highly stable when the momentum parameter ranges around $0.9$. These results illustrate that the performance of LFM is robust to the choices of its scaling parameter $\rho$ and momentum parameter $\beta$. According to these numerical experiments, we recommend a default setting of $\rho = 1$ and $\beta = 0.9$ for the practical implementations of LFM.

\subsection{Comparison with state-of-the-art learning-rate-free methods}

\begin{figure}[!tb]
	\centering
	\subfigure[Test accuracy, CIFAR-10]{
		\begin{minipage}[t]{0.33\linewidth}
			\centering
			\includegraphics[width=\linewidth]{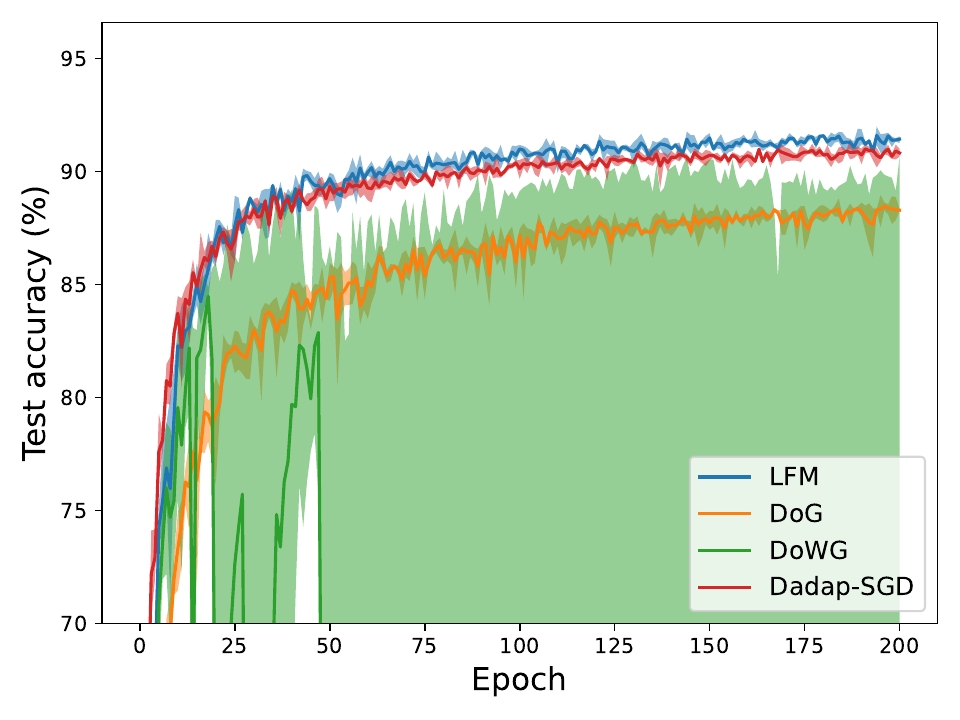}
			\label{Fig:Fig_curve_0_testacc}
		\end{minipage}%
	}%
	\subfigure[Test loss, CIFAR-10]{
		\begin{minipage}[t]{0.33\linewidth}
			\centering
			\includegraphics[width=\linewidth]{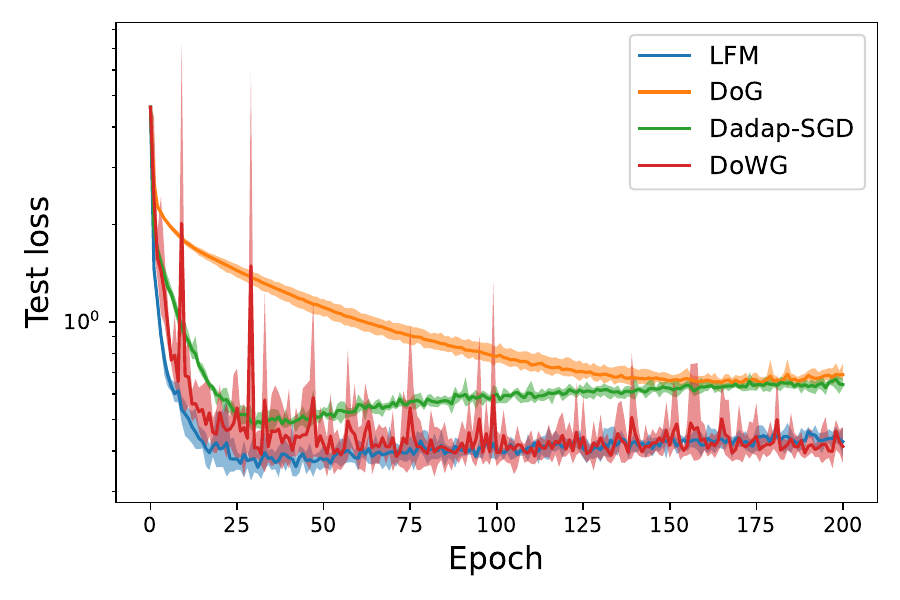}
			\label{Fig:Fig_curve_0_testloss}
		\end{minipage}%
	}%
	\subfigure[Training loss, CIFAR-10]{
		\begin{minipage}[t]{0.33\linewidth}
			\centering
			\includegraphics[width=\linewidth]{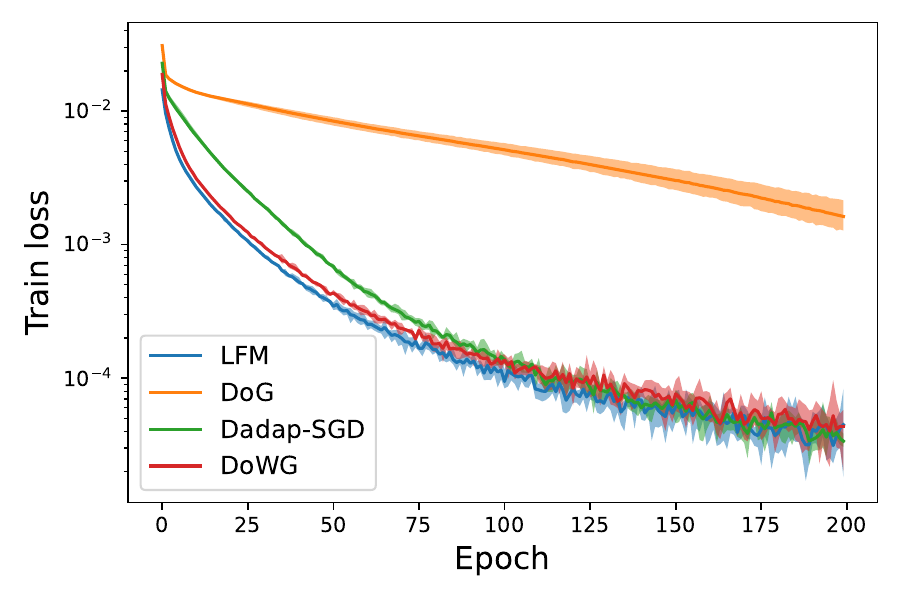}
			\label{Fig:Fig_curve_0_trainloss}
		\end{minipage}%
	}%

        \subfigure[Test accuracy, CIFAR-100]{
		\begin{minipage}[t]{0.33\linewidth}
			\centering
			\includegraphics[width=\linewidth]{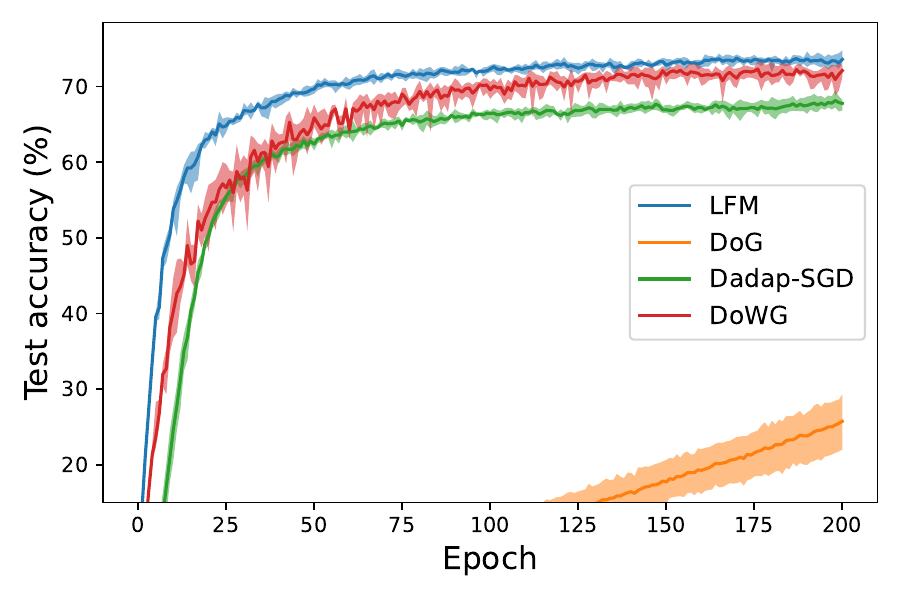}
			\label{Fig:Fig_curve_1_testacc}
		\end{minipage}%
	}%
	\subfigure[Test loss, CIFAR-100]{
		\begin{minipage}[t]{0.33\linewidth}
			\centering
			\includegraphics[width=\linewidth]{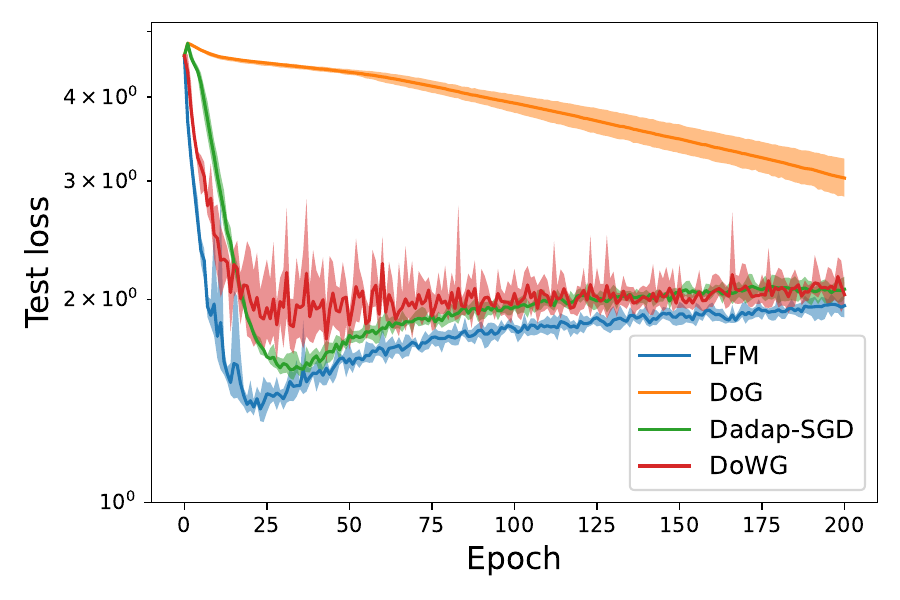}
			\label{Fig:Fig_curve_1_testloss}
		\end{minipage}%
	}%
	\subfigure[Training loss, CIFAR-100]{
		\begin{minipage}[t]{0.33\linewidth}
			\centering
			\includegraphics[width=\linewidth]{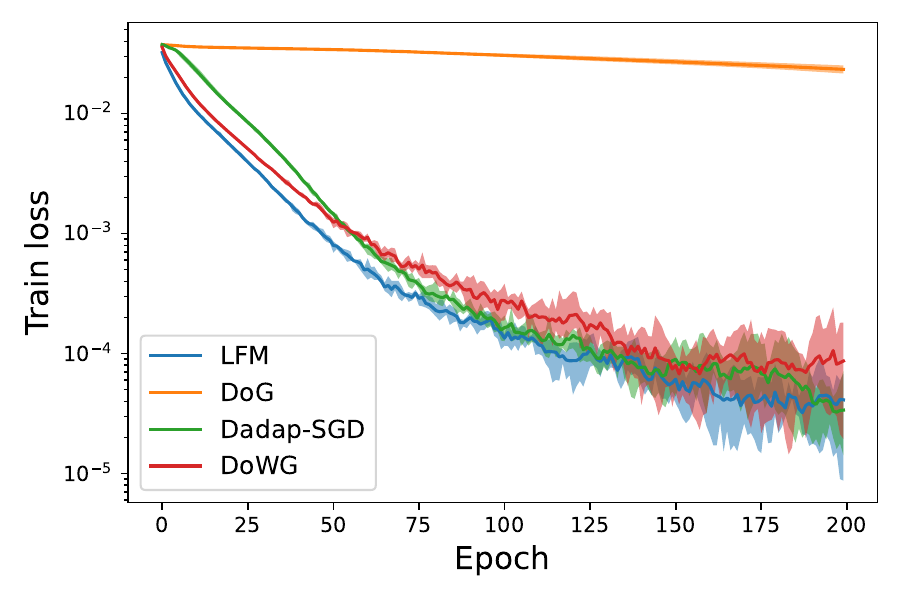}
			\label{Fig:Fig_curve_1_trainloss}
		\end{minipage}%
	}%
	\caption{Numerical results on applying LMF in Algorithm \ref{Alg:PFSGD}, DoG, DoWG, and D-adapted SGD for training ResNet50 over CIFAR datasets.}
	\label{Fig_Test_Resnet50}
\end{figure}



\begin{figure}[!tb]
	\centering
	\subfigure[Test accuracy, CIFAR-10]{
		\begin{minipage}[t]{0.33\linewidth}
			\centering
			\includegraphics[width=\linewidth]{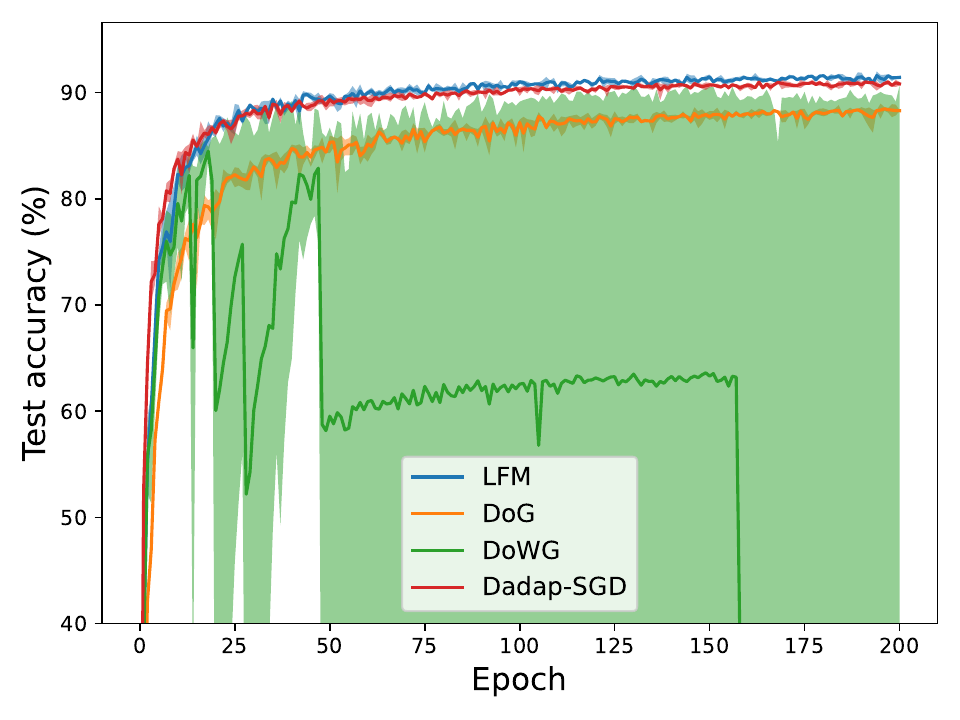}
			\label{Fig:vgg_0_testacc}
		\end{minipage}%
	}%
	\subfigure[Test loss, CIFAR-10]{
		\begin{minipage}[t]{0.33\linewidth}
			\centering
			\includegraphics[width=\linewidth]{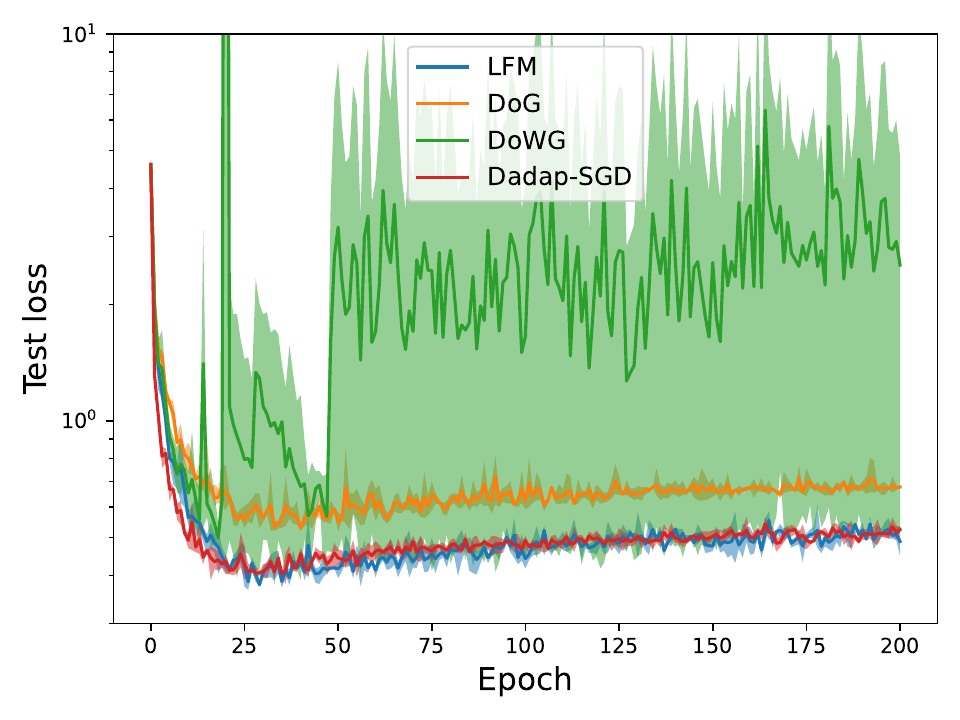}
			\label{Fig:vgg_0_testloss}
		\end{minipage}%
	}%
	\subfigure[Training loss, CIFAR-10]{
		\begin{minipage}[t]{0.33\linewidth}
			\centering
			\includegraphics[width=\linewidth]{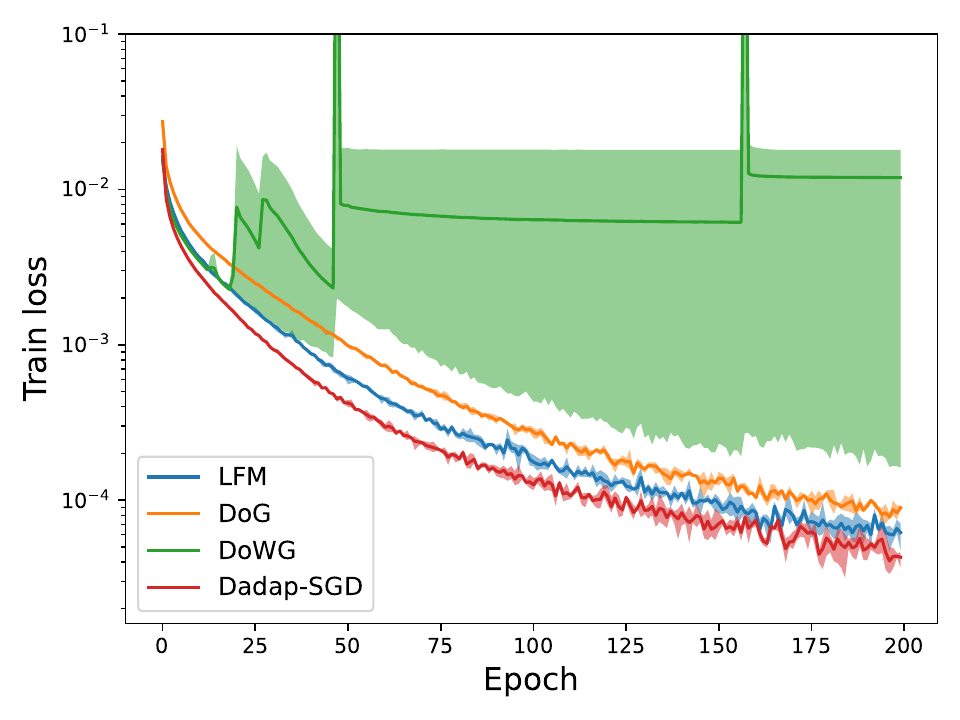}
			\label{Fig:vgg_0_trainloss}
		\end{minipage}%
	}%

        \subfigure[Test accuracy, CIFAR-100]{
		\begin{minipage}[t]{0.33\linewidth}
			\centering
			\includegraphics[width=\linewidth]{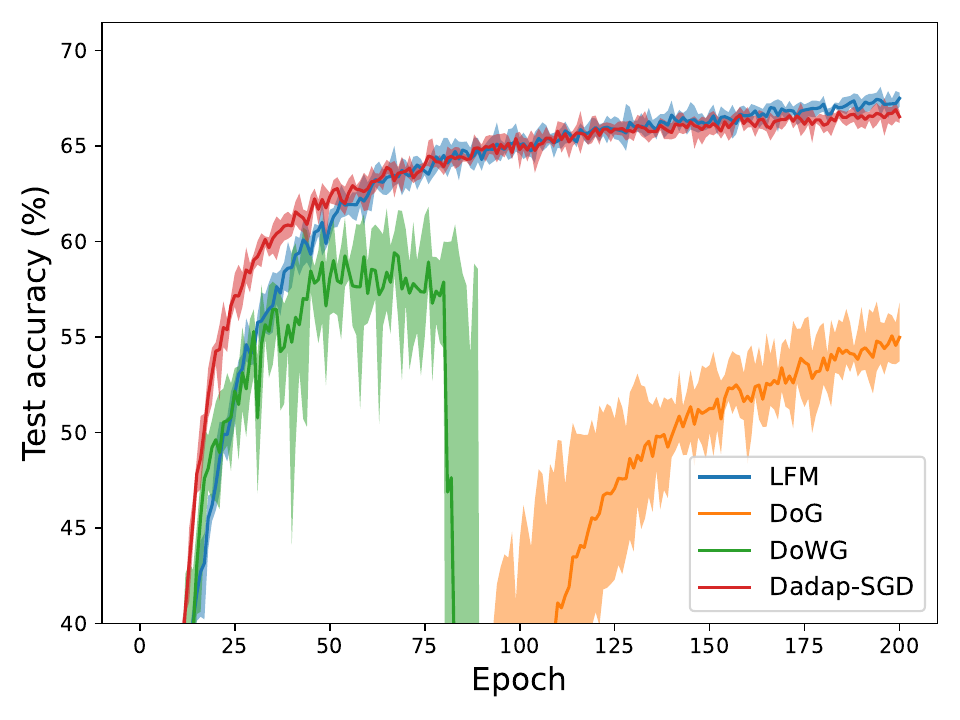}
			\label{Fig:vgg_1_testacc}
		\end{minipage}%
	}%
	\subfigure[Test loss, CIFAR-100]{
		\begin{minipage}[t]{0.33\linewidth}
			\centering
			\includegraphics[width=\linewidth]{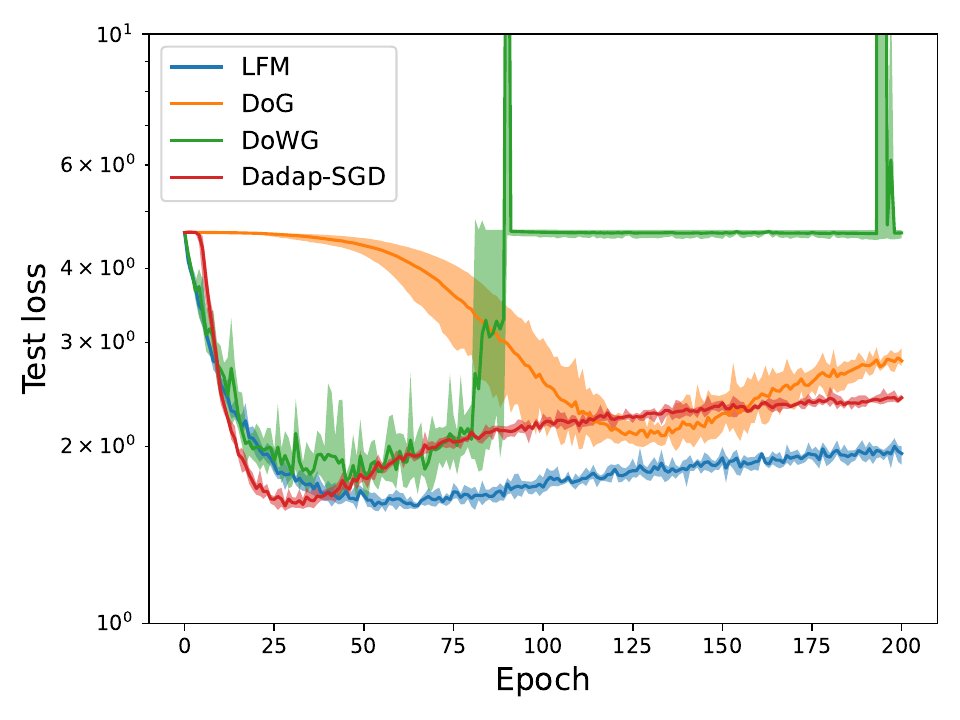}
			\label{Fig:vgg_1_testloss}
		\end{minipage}%
	}%
	\subfigure[Training loss, CIFAR-100]{
		\begin{minipage}[t]{0.33\linewidth}
			\centering
			\includegraphics[width=\linewidth]{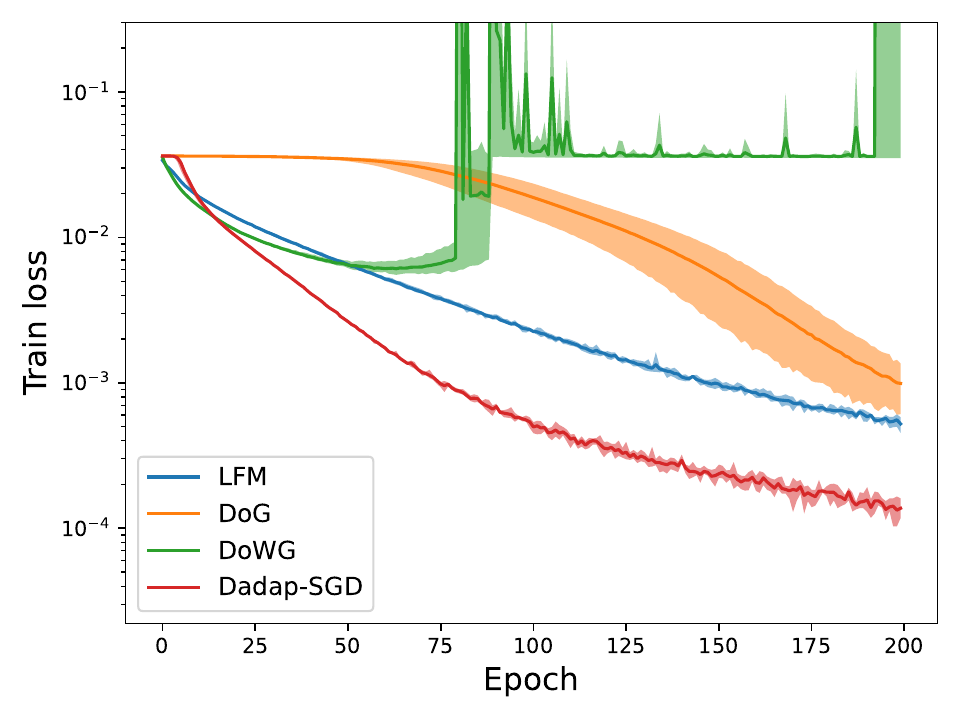}
			\label{Fig:vgg_1_trainloss}
		\end{minipage}%
	}%
	\caption{Numerical results on applying LMF in Algorithm \ref{Alg:PFSGD}, DoG, DoWG, and D-adapted SGD for training VGG-Net over CIFAR datasets.}
	\label{Fig_Test_VGG}
\end{figure}

\begin{figure}[!tb]
	\centering
	\subfigure[Test accuracy, CIFAR-10]{
		\begin{minipage}[t]{0.33\linewidth}
			\centering
			\includegraphics[width=\linewidth]{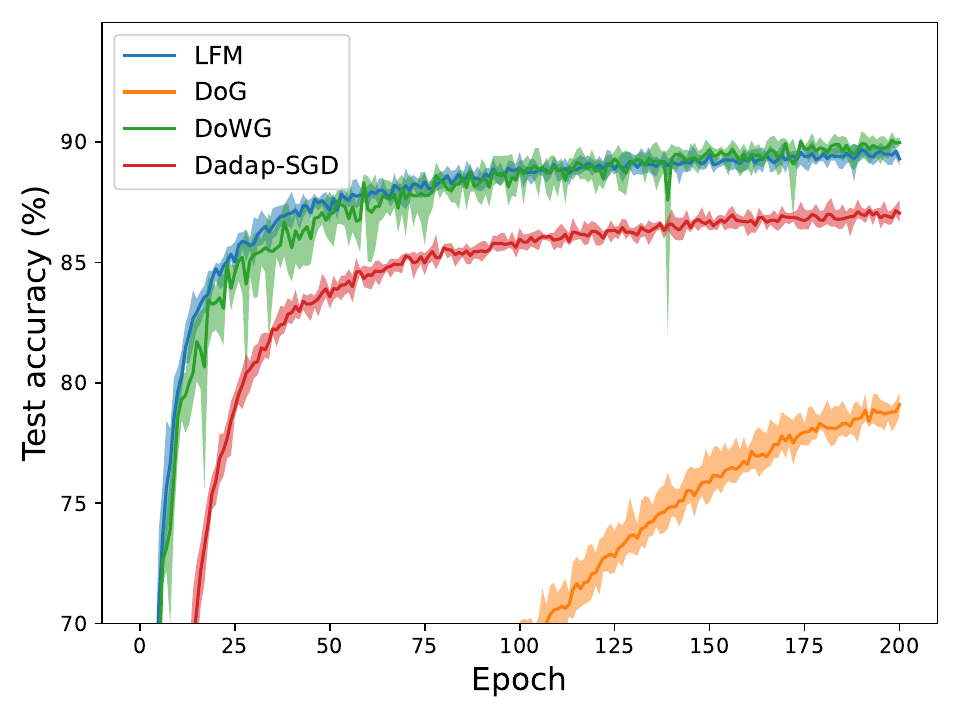}
			\label{Fig:mobilenet_0_testacc}
		\end{minipage}%
	}%
	\subfigure[Test loss, CIFAR-10]{
		\begin{minipage}[t]{0.33\linewidth}
			\centering
			\includegraphics[width=\linewidth]{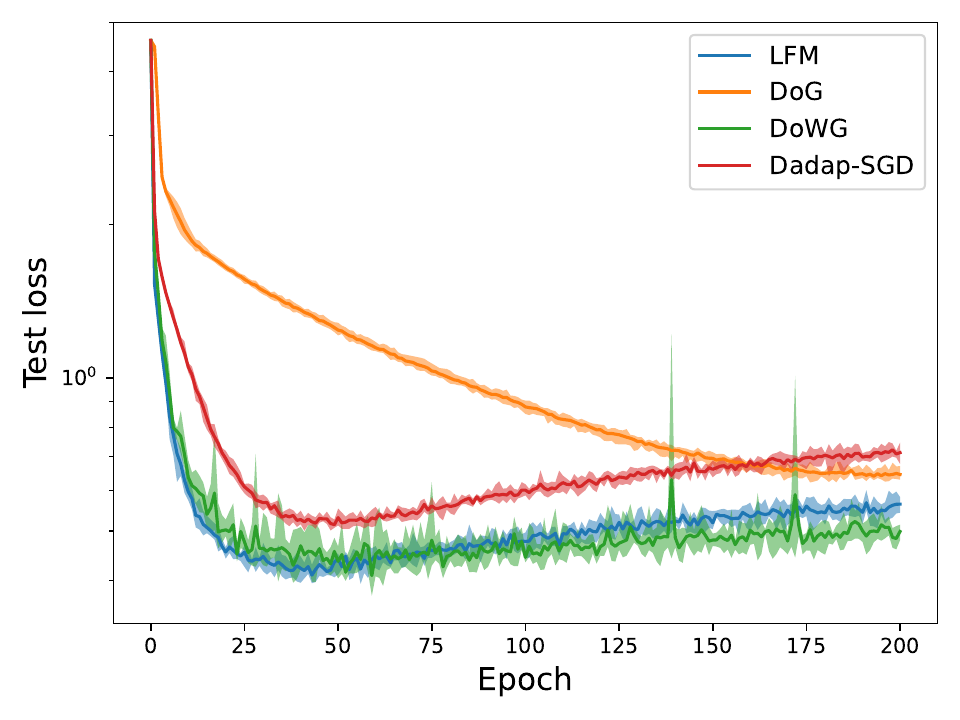}
			\label{Fig:mobilenet_0_testloss}
		\end{minipage}%
	}%
	\subfigure[Training loss, CIFAR-10]{
		\begin{minipage}[t]{0.33\linewidth}
			\centering
			\includegraphics[width=\linewidth]{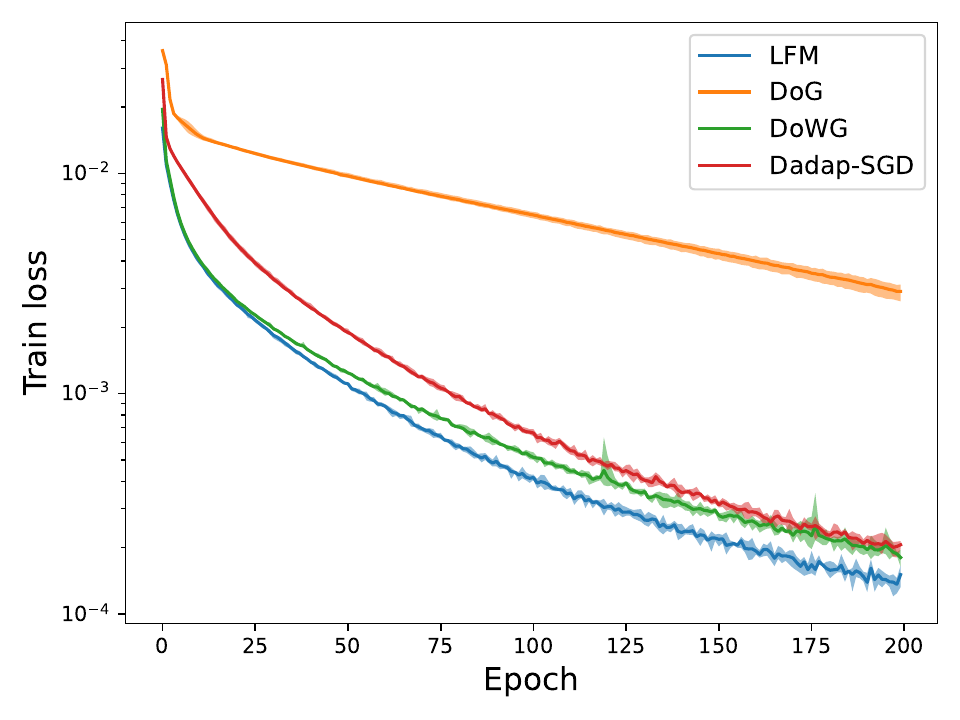}
			\label{Fig:mobilenet_0_trainloss}
		\end{minipage}%
	}%

        \subfigure[Test accuracy, CIFAR-100]{
		\begin{minipage}[t]{0.33\linewidth}
			\centering
			\includegraphics[width=\linewidth]{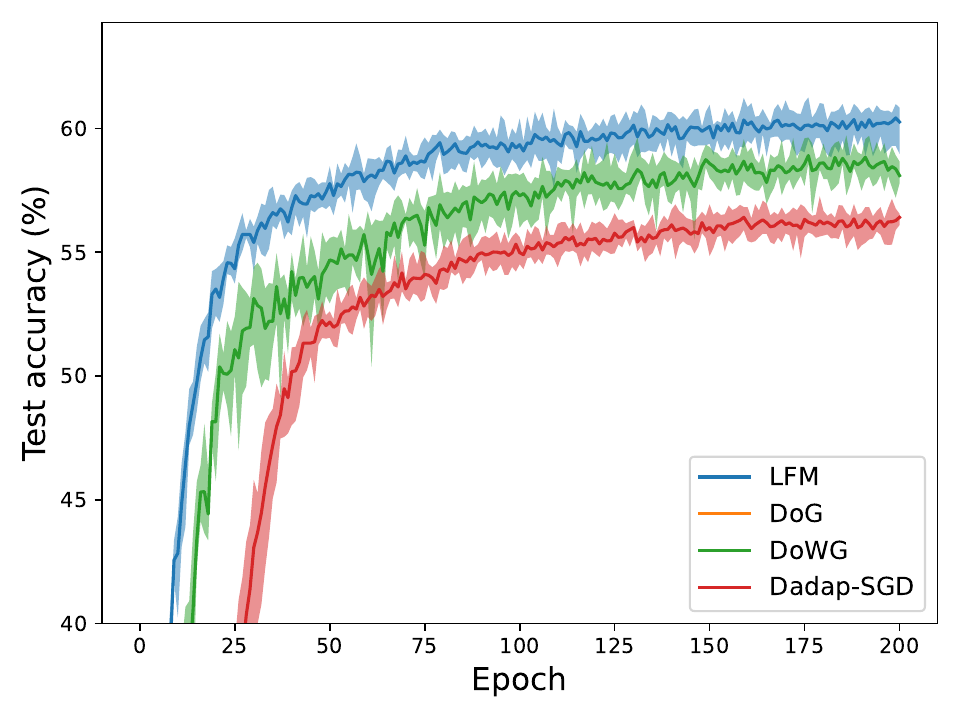}
			\label{Fig:mobilenet_1_testacc}
		\end{minipage}%
	}%
	\subfigure[Test loss, CIFAR-100]{
		\begin{minipage}[t]{0.33\linewidth}
			\centering
			\includegraphics[width=\linewidth]{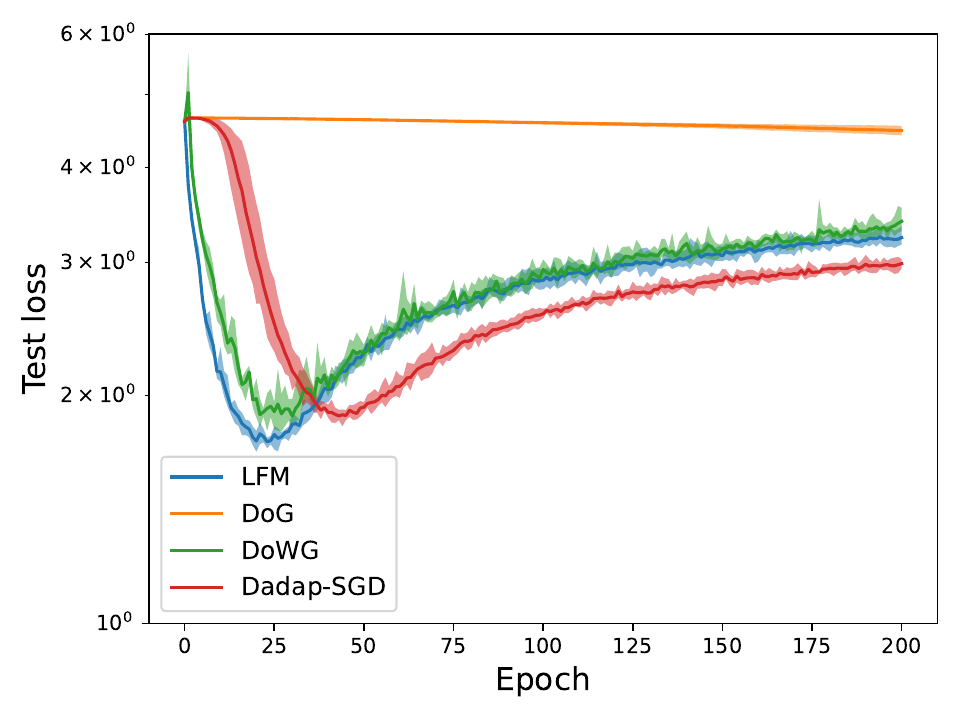}
			\label{Fig:mobilenet_1_testloss}
		\end{minipage}%
	}%
	\subfigure[Training loss, CIFAR-100]{
		\begin{minipage}[t]{0.33\linewidth}
			\centering
			\includegraphics[width=\linewidth]{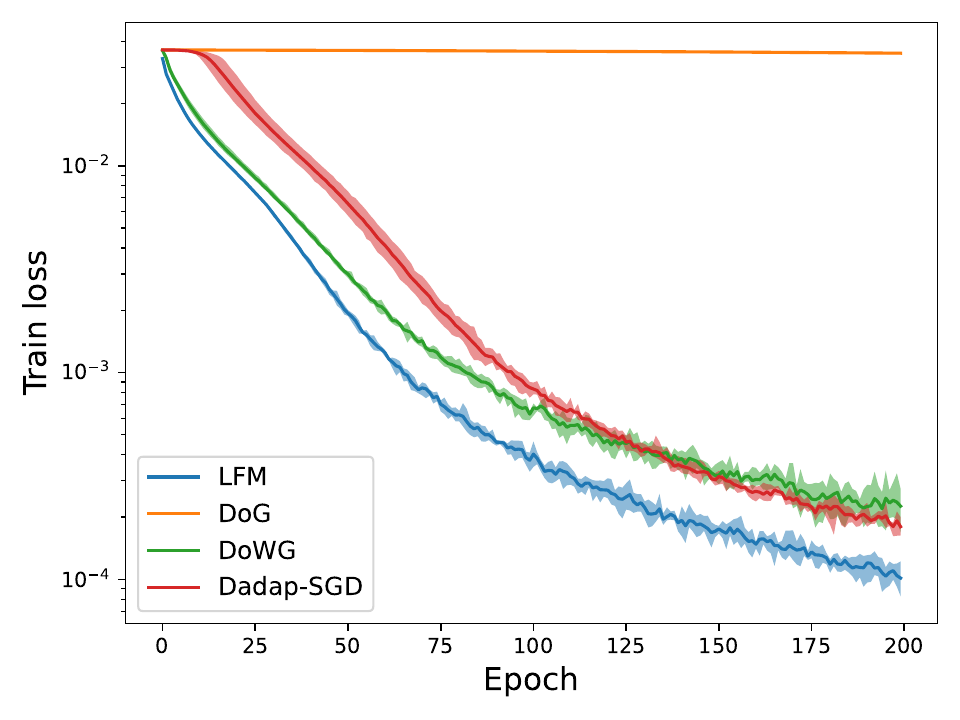}
			\label{Fig:mobilenet_1_trainloss}
		\end{minipage}%
	}%
	\caption{Numerical results on applying LMF in Algorithm \ref{Alg:PFSGD}, DoG, DoWG, and D-adapted SGD for training MobileNet over CIFAR datasets.}
	\label{Fig_Test_Mobilenet}
\end{figure}

\begin{figure}[!tb]
	\centering
	\subfigure[Top-1 test accuracy]{
		\begin{minipage}[t]{0.33\linewidth}
			\centering
			\includegraphics[width=\linewidth]{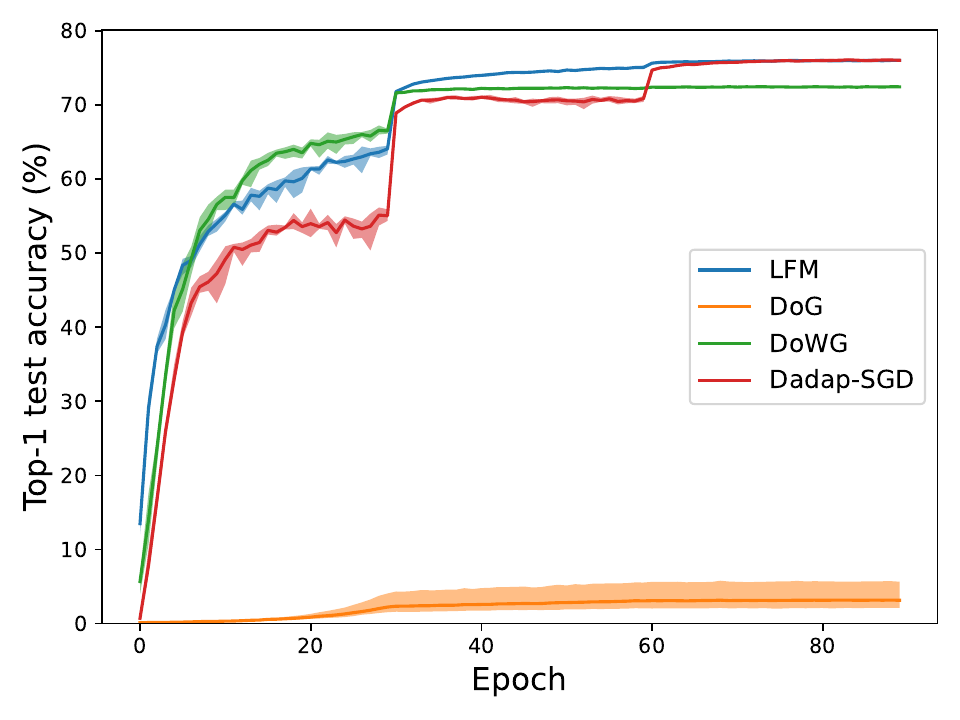}
			\label{Fig:imagenet_testacc_top1}
		\end{minipage}%
	}%
	\subfigure[Top-5 test accuracy]{
		\begin{minipage}[t]{0.33\linewidth}
			\centering
			\includegraphics[width=\linewidth]{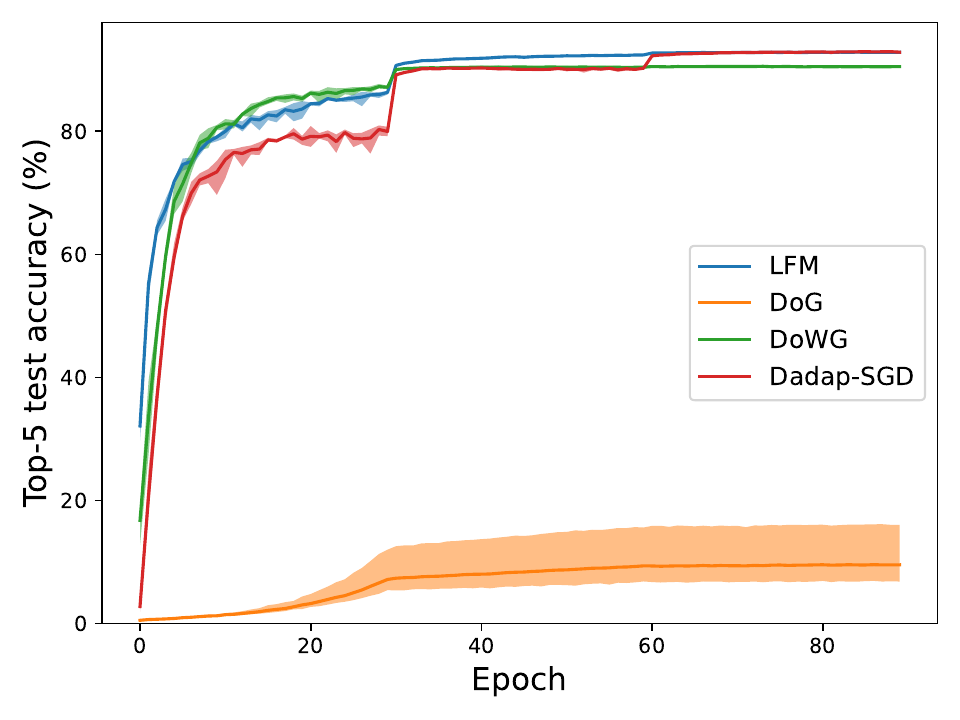}
			\label{Fig:imagenet_testacc_top5}
		\end{minipage}%
	}%
	\subfigure[Test loss]{
		\begin{minipage}[t]{0.33\linewidth}
			\centering
			\includegraphics[width=\linewidth]{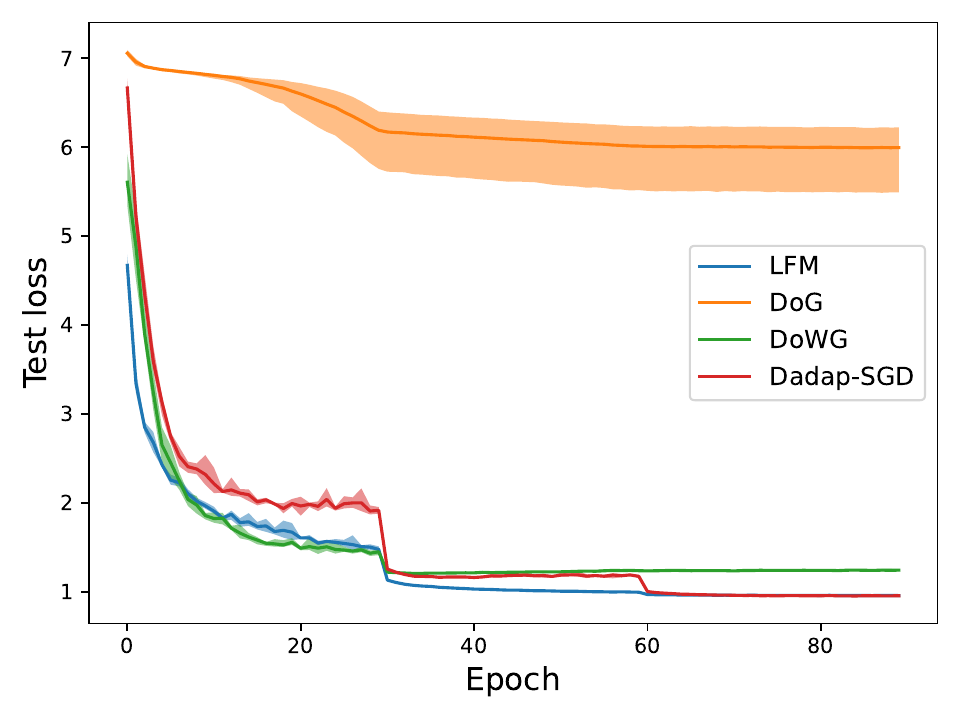}
			\label{Fig:imagenet_testloss}
		\end{minipage}%
	}%

        \subfigure[Top-1 train accuracy]{
		\begin{minipage}[t]{0.33\linewidth}
			\centering
			\includegraphics[width=\linewidth]{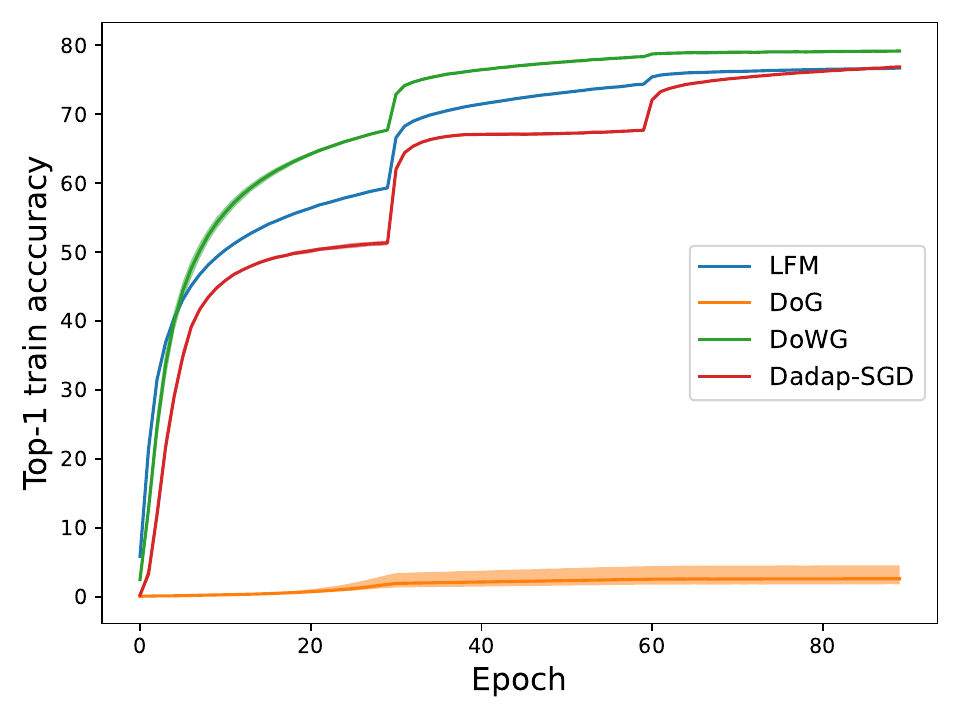}
			\label{Fig:imagenet_trainacc_top1}
		\end{minipage}%
	}%
	\subfigure[Top-5 train accuracy]{
		\begin{minipage}[t]{0.33\linewidth}
			\centering
			\includegraphics[width=\linewidth]{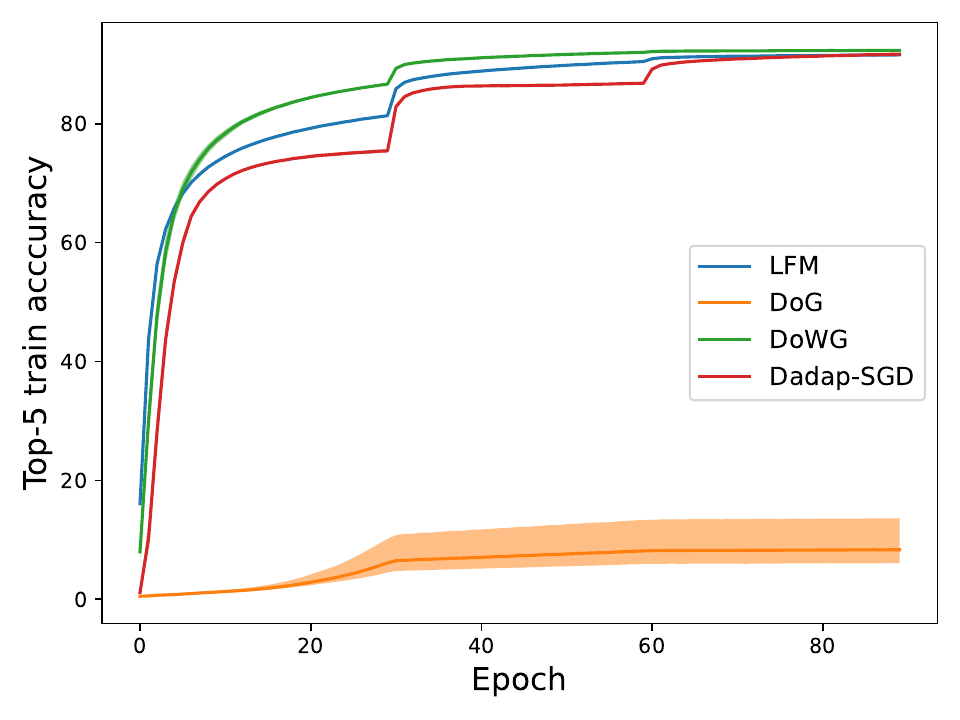}
			\label{Fig:imagenet_trainacc_top5}
		\end{minipage}%
	}%
	\subfigure[Train loss]{
		\begin{minipage}[t]{0.33\linewidth}
			\centering
			\includegraphics[width=\linewidth]{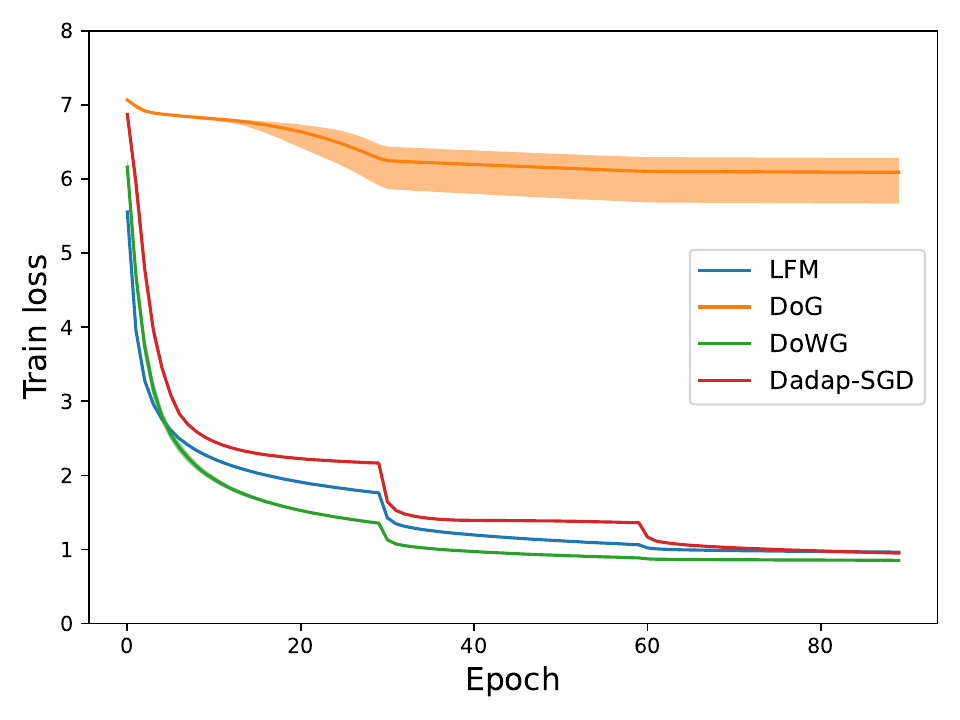}
			\label{Fig:imagenet_trainloss}
		\end{minipage}%
	}%
	\caption{Numerical results on applying LMF in Algorithm \ref{Alg:PFSGD}, DoG, DoWG, and D-adapted SGD for training ResNet-50 over Imagenet datasets.}
	\label{Fig_Test_Imagenet}
\end{figure}

To evaluate the numerical efficiency of LFM, we compare the performance of our proposed method in Algorithm \ref{Alg:PFSGD} with the state-of-the-art learning-rate-free methods, including DoG \cite{ivgi2023dog}, DoWG \cite{khaled2023dowg}, and D-adapted SGD  \cite{defazio2023learning} (abbreviated as Dadap-SGD). For all methods under consideration, we standardized the momentum parameter at $0.9$ and retained default settings for other parameters. The models are trained over 200 epochs with a batch size of $128$, and each test instance is executed five times with different random seeds. In each test instance, all the compared methods employ the same random seed, and all the initial points are generated by the default initialization function in PyTorch. Furthermore, we extended our evaluation to compare the performance of these methods across a range of neural network architectures, including ResNet50, ResNet18, VGG-Net, and MobileNet. This additional analysis provides a broader perspective on the adaptability and performance of each method within diverse computational environments.

Figures \ref{Fig_Test_Resnet50} - \ref{Fig_Test_Imagenet} exhibit the curves of the test accuracy, training loss, and test loss of all the compared methods in our experiments. The results depicted in Figure \ref{Fig_Test_Resnet50} indicate that LFM outperforms both DoG and Dadap-SGD, while slightly surpassing DoWG in terms of test error and accuracy when training ResNet50 on CIFAR datasets. Moreover, in Figure \ref{Fig_Test_VGG} and \ref{Fig_Test_Mobilenet}, DoWG exhibits unstable performance and may fail in some test instances (e.g.,  training VGG-Net on  CIFAR-100  in Figure \ref{Fig:vgg_1_testacc}-\ref{Fig:vgg_1_trainloss}). In contrast, the performance of LFM is stable on all the 
test instances. Furthermore, we present the results of training ResNet50 on the ImageNet dataset \cite{deng2009imagenet}, as shown in Figure \ref{Fig_Test_Imagenet}, LFM outperforms DoG and DoWG while exhibiting slightly better performance than Dadap-SGD in the aspect of test accuracy and test loss. Moreover, in terms of training loss and training accuracy, LFM is comparable to DoWG and Dadap-SGD,  and significantly outperforms DoG. These empirical results underscore the high efficiency and robustness of our proposed learning-rate-free method LFM in Algorithm \ref{Alg:PFSGD}.

\section{Conclusion}

In this paper, we first propose a general framework for developing learning-rate-free momentum SGD, where the learning rates are adaptively computed from the historical subgradients and iterates. Under reshuffling schemes, we establish the convergence properties for our proposed framework \eqref{Eq_Framework}, in the sense that any cluster point of sequence $\{\xk\}$ lies in the subset $\{x \in \Rn: 0 \in \D_f(x)\}$. Furthermore, under mild conditions,  we prove the global stability for \eqref{Eq_Framework_determin}, which is a deterministic version of \eqref{Eq_Framework}. 

We show that our proposed framework \eqref{Eq_Framework} provides convergence guarantees for DoG \cite{ivgi2023dog} and DoWG \cite{khaled2023dowg} in training nonsmooth neural networks. Moreover, based on our proposed framework \eqref{Eq_Framework}, we introduce a novel subgradient method named LFM in Algorithm \ref{Alg:PFSGD}, which is the first momentum-accelerated learning-rate-free SGD with guaranteed convergence in training nonsmooth neural networks. Preliminary numerical experiments illustrate that our proposed LFM outperforms both DoG and Dadap-SGD in terms of efficiency, while exhibiting enhanced robustness and comparable test accuracy as DoWG. These preliminary numerical experiments further illustrate the promising potential of our proposed framework \eqref{Eq_Framework}.


%
%

\bibliographystyle{spmpsci}      
\bibliography{ref.bib}   


\end{document}